\documentclass[a4paper,12pt]{amsart}

\usepackage{float}
\usepackage{euscript,eufrak,verbatim}
\usepackage{graphicx}
\usepackage[usenames]{color}
\usepackage[colorlinks,linkcolor=red,anchorcolor=blue,citecolor=blue]{hyperref}
\usepackage{amsmath}
\usepackage{amsthm}
\usepackage{graphicx} 
\usepackage{amssymb} 
\usepackage{amsfonts}

\usepackage{mathrsfs}
\usepackage{amscd}

\usepackage{tikz-cd}

\usepackage{bbm}

\usepackage{booktabs,tabularx}

\makeatletter
\newcommand{\svdots}{%
  \vbox{\fontsize{\sf@size}{\sf@size pt}\linespread{0.3}\selectfont
    \kern0.2\baselineskip
    \hbox{.}\hbox{.}\hbox{.}%
    \kern0.1\baselineskip
  }%
}
\makeatother

\theoremstyle{plain}
\newtheorem{main theorem}{Main Theorem}
\newtheorem{theorem}{Theorem}[section]
\newtheorem{lemma}[theorem]{Lemma}

\newtheorem{proposition}[theorem]{Proposition}
\newtheorem{claim}[theorem]{Claim}

\newtheorem{lemma-definition}[theorem]{Lemma-Definition}
\theoremstyle{definition}

\newtheorem{remark}[theorem]{Remark}

\makeatletter 
\@addtoreset{equation}{section}
\numberwithin{equation}{section}

\newcommand{\norm}[1]{\left\lVert#1\right\rVert}

\newcommand{\diam}{\mathrm{Diam}}

\newcommand{\h}{h_{\mathrm{top}}}
\newcommand{\FHh}{h_{\mathrm{FH}}}

\newcommand{\hinf}{\underline{h}_{\mathrm{top}}}

\title[Comparison of two approaches to weighted topological entropy]
{Comparison of two approaches to weighted topological entropy}

\author{Masaki Tsukamoto}

\address[Masaki Tsukamoto]
{Department of Mathematics, Kyoto University, Kyoto 606-8502, Japan}

\email{tsukamoto@math.kyoto-u.ac.jp}

\begin{document}



\thanks{M.T. was supported by JSPS KAKENHI JP25K06974.}

\begin{abstract}
We compare the Feng--Huang and covering definitions of weighted
topological entropy for equivariant continuous maps between
dynamical systems.
The two quantities agree on the whole space by their variational
principles, but need not agree on non-invariant subsets.
We prove that the Feng--Huang
entropy is bounded above by the covering entropy on every subset,
without assuming invariance.
The proof does not rely on measure theory.
Our main result provides a new perspective on
the two relative weighted variational principles.
\end{abstract}

\maketitle

\section{Introduction}  \label{section: Introduction}

\subsection{Background and main result}
\label{subsection: Background and main result}

The main purpose of this paper is to compare two approaches to weighted
topological entropy for arbitrary subsets of a dynamical system, with
particular emphasis on the non-invariant setting.

Motivated by the fractal geometry of self-affine carpets and sponges,
Feng and Huang \cite{Feng--Huang} introduced weighted topological entropy.
It extends the classical theory of topological entropy from a single
dynamical system to the setting of a factor map between dynamical systems.
Their definition follows Bowen's approach \cite{Bowen} to topological
entropy for noncompact sets, in which the construction parallels that
of Hausdorff dimension.

Following the pioneering work of Feng and Huang, the author introduced
a different definition in \cite{Tsukamoto}.
This definition builds on ideas of Barral and Feng
\cite{Barral--Feng_arXiv, Barral--Feng} and is closer in spirit to the
standard definition of topological entropy.
We begin by recalling the two definitions.

A pair $(X,T)$ is called a \textbf{dynamical system} if $X$ is a compact
metric space and $T\colon X\to X$ is a continuous map.
Let $(X,T)$ and $(Y,S)$ be two dynamical systems, and let
$\pi\colon X\to Y$ be an equivariant continuous map, meaning that
$\pi\circ T=S\circ\pi$.
We often write $\pi\colon(X,T)\to(Y,S)$ for such a map.
Fix a real number $0\leq w\leq1$.
We consider two quantities in this setting: Feng--Huang's weighted
topological entropy $\FHh^w(X,T)$ and the covering weighted topological
entropy $\h^w(X,T)$.
We first recall the Feng--Huang definition, using notation slightly
different from that in \cite{Feng--Huang}.

Let $\mathbf{d}$ and $\mathbf{d}^\prime$ be metrics on $X$ and $Y$,
respectively.
For a natural number $n$, define the \textbf{Bowen metrics}
$\mathbf{d}_n$ on $X$ and $\mathbf{d}^\prime_n$ on $Y$ by
\[
\mathbf{d}_n(x,x^\prime)
=
\max_{0\leq k<n}\mathbf{d}(T^kx,T^kx^\prime),
\quad
\mathbf{d}^\prime_n(y,y^\prime)
=
\max_{0\leq k<n}\mathbf{d}^\prime(S^ky,S^ky^\prime).
\]
For $x,x^\prime\in X$, we also set
\[
\mathbf{d}^w_n(x,x^\prime)
=
\max\left(
\mathbf{d}_{\lceil wn\rceil}(x,x^\prime),
\mathbf{d}^\prime_n(\pi(x),\pi(x^\prime))
\right).
\]
Here we use the convention $\mathbf{d}_0(x,x^\prime)=0$ when
$\lceil wn\rceil=0$.
For $\varepsilon>0$, the \textbf{$w$-weighted Bowen ball} is defined by
\[
B^w_n(x,\varepsilon)
=
\{x^\prime\in X\mid\mathbf{d}^w_n(x,x^\prime)<\varepsilon\}.
\]

Let $\Omega$ be a subset of $X$, not necessarily $T$-invariant.
Let $N$ be a natural number, and let $s\geq0$.
We consider coverings of $\Omega$ by at most countably many
$w$-weighted Bowen balls $B^w_{n_i}(x_i,\varepsilon)$ with $n_i\geq N$:
\begin{equation}
\label{eq: covering of weighted Bowen balls}
\Omega\subset\bigcup_i B^w_{n_i}(x_i,\varepsilon),
\quad (x_i\in X,\ n_i\geq N).
\end{equation}
Set
\[
\Lambda^{w,s}_{N,\varepsilon}(\Omega)
:=\inf\sum_i e^{-sn_i},
\]
where the infimum is taken over all at most countable coverings
satisfying \eqref{eq: covering of weighted Bowen balls}.
This quantity is nondecreasing in $N$, so we define
\[
\Lambda^{w,s}_\varepsilon(\Omega)
:=\lim_{N\to\infty}\Lambda^{w,s}_{N,\varepsilon}(\Omega).
\]
As in the definition of Hausdorff dimension, there is a unique
critical value $\FHh^w(\Omega,T,\varepsilon)$ such that
\[
\Lambda^{w,s}_\varepsilon(\Omega)
=
\begin{cases}
\infty & (s<\FHh^w(\Omega,T,\varepsilon)),\\
0 & (s>\FHh^w(\Omega,T,\varepsilon)).
\end{cases}
\]
We then define \textbf{Feng--Huang's $w$-weighted topological entropy
of $\Omega$} by
\[
\FHh^w(\Omega,T)
=\lim_{\varepsilon\to0}\FHh^w(\Omega,T,\varepsilon).
\]
When we wish to emphasize the underlying map
$\pi\colon(X,T)\to(Y,S)$, we also write $\FHh^w(\Omega,\pi,T)$.
We adopt the convention that $\FHh^w(\Omega,T)=-\infty$ if
$\Omega=\emptyset$.

When $\Omega=X$, Feng and Huang \cite[Theorem 1.4]{Feng--Huang}
established the variational principle
\begin{equation}
\label{eq: Feng--Huang variational principle}
\FHh^w(X,T)
=
\sup_{\mu\in\mathscr{M}^T(X)}
\left\{wh_\mu(X,T)+(1-w)h_{\pi_*\mu}(Y,S)\right\},
\end{equation}
where $\mathscr{M}^T(X)$ is the set of all $T$-invariant Borel
probability measures on $X$, and $h_\mu(X,T)$ and
$h_{\pi_*\mu}(Y,S)$ denote the Kolmogorov--Sinai entropies of
$\mu$ and its push-forward $\pi_*\mu$, respectively.

We next recall the definition introduced in \cite{Tsukamoto},
which is based on a two-step covering procedure.
Let $\Omega\subset X$, $n\in\mathbb{N}$, and $\varepsilon>0$.
The \textbf{$\varepsilon$-covering number}
$\#(\Omega,\mathbf{d}_n,\varepsilon)$ of $(\Omega,\mathbf{d}_n)$
is the smallest integer $k$ for which there are open subsets
$U_1,\dots,U_k$ of $X$ satisfying
$\Omega\subset U_1\cup\dots\cup U_k$ and
$\diam(U_i,\mathbf{d}_n)<\varepsilon$ for every $1\leq i\leq k$.
The \textbf{$w$-weighted two-step $\varepsilon$-covering number}
of $\Omega$ is defined by
\begin{equation*}
\begin{split}
&\#^w\left(\Omega,n,\varepsilon\right)\\
&=\inf\left\{
\sum_{j=1}^{\ell}
\left(\#\left(\Omega\cap\pi^{-1}(V_j),\mathbf{d}_n,\varepsilon\right)\right)^w
\ \middle|\
\parbox{3in}{\centering
$V_1,\dots,V_\ell$ are open subsets of $Y$ satisfying
$\pi(\Omega)\subset V_1\cup\dots\cup V_\ell$ and
$\diam(V_j,\mathbf{d}^\prime_n)<\varepsilon$ for every $j$}
\right\}.
\end{split}
\end{equation*}
Here we set
$\left(\#\left(\Omega\cap\pi^{-1}(V_j),\mathbf{d}_n,\varepsilon\right)\right)^w=0$
whenever $\Omega\cap\pi^{-1}(V_j)=\emptyset$.
We define the \textbf{covering $w$-weighted topological entropy}
of $\Omega$ by
\begin{equation}
\label{eq: covering weighted topological entropy}
\h^w(\Omega,T)
=
\lim_{\varepsilon\to0}
\left(
\limsup_{n\to\infty}
\frac{\log\#^w\left(\Omega,n,\varepsilon\right)}{n}
\right).
\end{equation}
As before, we also write $\h^w(\Omega,\pi,T)$ 
when we wish to
emphasize the underlying map $\pi\colon(X,T)\to(Y,S)$.
We adopt the convention that $\h^w(\Omega,T)=-\infty$ if
$\Omega=\emptyset$.

Some readers may wonder what happens if the limit superior in
\eqref{eq: covering weighted topological entropy} is replaced by
the limit inferior.
More precisely, we may consider
\begin{equation}
\label{eq: covering weighted topological entropy, liminf}
\hinf^w(\Omega,T)
=
\lim_{\varepsilon\to0}
\left(
\liminf_{n\to\infty}
\frac{\log\#^w\left(\Omega,n,\varepsilon\right)}{n}
\right),
\end{equation}
and ask which of the two quantities $\h^w(\Omega,T)$ and
$\hinf^w(\Omega,T)$ is more useful.
Somewhat surprisingly, the limsup version $\h^w(\Omega,T)$ is
better suited to our purposes, whereas the liminf version
$\hinf^w(\Omega,T)$ is of rather limited use in the present paper;
see Remark \ref{remark: main theorem} (2) and \S \ref{section: example} below.

If $\Omega$ is $T$-invariant (in particular, if $\Omega=X$), then
$\log\#^w\left(\Omega,n,\varepsilon\right)$ is subadditive in $n$.
Consequently, the limit
\[
\lim_{n\to\infty}
\frac{\log\#^w\left(\Omega,n,\varepsilon\right)}{n}
\]
exists, and $\h^w(\Omega,T)=\hinf^w(\Omega,T)$.

For $\Omega=X$, the author established the corresponding variational
principle in \cite[Theorem 1.3]{Tsukamoto}:
\begin{equation}
\label{eq: variational principle for covering weighted entropy}
\h^w(X,T)
=
\sup_{\mu\in\mathscr{M}^T(X)}
\left\{wh_\mu(X,T)+(1-w)h_{\pi_*\mu}(Y,S)\right\}.
\end{equation}

Although $\FHh^w(X,T)$ and $\h^w(X,T)$ are defined in very different
ways, the two variational principles
\eqref{eq: Feng--Huang variational principle} and
\eqref{eq: variational principle for covering weighted entropy}
show that they coincide:
\begin{equation}
\label{eq: coincidence of two weighted topological entropy}
\FHh^w(X,T)=\h^w(X,T).
\end{equation}
This identity is quite surprising in view of the definitions.
It also has useful applications.
For example, the Hausdorff dimension formula for Bedford--McMullen
carpets \cite{Bedford, McMullen} follows readily from
\eqref{eq: coincidence of two weighted topological entropy};
see \cite[Example 1.6]{Tsukamoto}.
In this sense, \eqref{eq: coincidence of two weighted topological entropy}
can be viewed as a \textit{topological generalization} of the
Hausdorff dimension formula for Bedford--McMullen carpets.

We now turn to arbitrary subsets $\Omega\subset X$ and compare
$\FHh^w(\Omega,T)$ with $\h^w(\Omega,T)$.
This setting arises naturally in the study of relative weighted
variational principles \cite{Yin,Alibabaei--Tsukamoto}.
A primary motivation for the present work is to gain a better
understanding of these relative versions; see
\S\ref{subsection: relative weighted variational principles} below.

For a general subset $\Omega\subset X$, the equality
\[
\FHh^w(\Omega,T)=\h^w(\Omega,T)
\]
need not hold.
Nevertheless, one direction of this equality remains valid in
full generality.
Our main result is the following.

\begin{theorem}
\label{theorem: main theorem}
Let $0\leq w\leq1$, and let
$\pi\colon(X,T)\to(Y,S)$ be an equivariant continuous map between dynamical systems.
Then, for every subset $\Omega\subset X$, we have
\[
\FHh^w(\Omega,T)\leq\h^w(\Omega,T).
\]
\end{theorem}

Here $\Omega$ is not assumed to be either $T$-invariant or closed; 
see also Remark \ref{remark: main theorem} (1) below.

An application of Theorem \ref{theorem: main theorem} is discussed in
\S\ref{subsection: relative weighted variational principles}.
Before turning to this application, we highlight a methodological
feature of the proof: it does not rely on measure theory.
Our aim in adopting this approach is to isolate the topological and
combinatorial structure underlying the theorem.

The motivation for this approach comes from the contrast between
the definitions and the existing proof of their equivalence.
Both $\FHh^w(X,T)$ and $\h^w(X,T)$ are topological invariants, yet
the proof in \cite{Tsukamoto} of the identity
\[
\FHh^w(X,T)=\h^w(X,T)
\]
relies on the two variational principles
\eqref{eq: Feng--Huang variational principle} and
\eqref{eq: variational principle for covering weighted entropy}.
It is therefore natural to ask whether this identity admits
a direct proof that does not use measure theory.
This question was posed explicitly in \cite[Problem 1.5]{Tsukamoto}.

Our proof of Theorem \ref{theorem: main theorem} provides a partial
answer to this question.
It shows that one direction of the identity, namely,
\[
\FHh^w(X,T)\leq\h^w(X,T),
\]
can be established by purely topological means.

We first prove Theorem \ref{theorem: main theorem} in the special
case $\Omega=X$ in
\S\ref{section: proof of the main theorem for Omega = X},
and then prove the general case in
\S\ref{section: proof of the general case of the main theorem}.
The proof of the general case does not rely on the argument
for $\Omega=X$.
We nevertheless devote a separate section to this special case,
since its proof is substantially simpler and interesting in its
own right.
Moreover, examining this simpler setting helps clarify where
the essential difficulties of the general case arise.

We would like to clarify what we mean by 
\lq\lq{}proofs without using measure theory\rq\rq{}.
In the case $\Omega=X$, 
we will give a proof that does not use any 
measure-theoretic notions or arguments.
Thus, it is \lq\lq{}measure-free\rq\rq{} in the literal sense.
For a general subset $\Omega\subset X$, 
we will use some notions of elementary probability theory, 
such as Shannon entropy. 
However, this is done only for convenience of exposition. 
Every probability measure appearing in our proof is finitely supported 
and has the form
\[
\mu=\sum_{i=1}^n a_i\delta_{x_i},
\qquad
a_i\geq0,\quad \sum_{i=1}^n a_i=1,
\]
where $\{x_1,\dots,x_n\}$ is a finite subset of $X$, 
$\delta_{x_i}$ denotes the Dirac measure at $x_i$, 
and $(a_1,\dots,a_n)\in\mathbb R^n$ is a probability vector. 
In particular, we never take any limit of these measures.
Thus, the use of probability measures here is essentially finite-dimensional: 
all the relevant information is encoded by a finite subset of $X$ 
together with a finite-dimensional probability vector. 
If desired, the entire proof could therefore be reformulated in terms of finite subsets of $X$, 
probability vectors, 
and elementary finite-dimensional arguments, without introducing measures at all. 
Such a reformulation, however, would be less intuitive and considerably harder to read. 
For this reason, we use the language of elementary probability theory in the proof.

\begin{remark} \label{remark: main theorem}
We make a few remarks to clarify the statement of Theorem \ref{theorem: main theorem}.
  \begin{enumerate}
   \item[(1)] In Theorem \ref{theorem: main theorem}, 
   we assume neither that $\Omega\subset X$ is $T$-invariant 
   nor that it is closed; $\Omega$ is an arbitrary subset of $X$.
   At first sight, allowing such completely general subsets may seem somewhat unusual or unnecessarily general.
   However, as far as closedness is concerned, this generality is largely inessential.
   Let us denote the closure of $\Omega$ by $\overline{\Omega}$. 
   It is not difficult to see that 
   \[ \FHh^w(\Omega, T) \leq \FHh^w\left(\overline{\Omega}, T\right), \quad 
       \h^w(\Omega, T) = \h^w\left(\overline{\Omega}, T\right). \] 
    In particular, if we prove the inequality $\FHh^w\left(\overline{\Omega}, T\right) \leq \h^w\left(\overline{\Omega}, T\right)$,
    then we also have $\FHh^w(\Omega, T) \leq \h^w(\Omega, T)$.
    Thus, although the theorem is stated for arbitrary subsets of $X$, 
    allowing nonclosed subsets does not introduce any essentially new difficulty.
    We will not use this reduction in the actual proof, 
    but it may help make clear that the formulation for arbitrary subsets is not substantially more general than the closed case.
    In contrast, the non-$T$-invariant setting is a genuinely new feature of the problem 
    and gives rise to difficulties absent in the invariant case.
    \item[(2)] The two quantities $\FHh^w(\Omega,T)$ and $\h^w(\Omega,T)$ need not be equal.
    Here we clarify the possible order relations among the following three quantities:
    \[ \FHh^w(\Omega, T), \quad \h^w(\Omega, T), \quad \hinf^w(\Omega, T). \]
    We always have 
    \[ \FHh^w(\Omega, T)  \leq \h^w(\Omega, T), \quad \hinf^w(\Omega, T) \leq \h^w(\Omega, T). \]
    The first inequality is Theorem \ref{theorem: main theorem}, while the second follows immediately from the definitions.
    It turns out that, in general, there are no further universal order relations among these three quantities.
    First, consider the case
    \[ X = Y= \mathbb{R}/\mathbb{Z}, \quad T(x)= S(x) = 2x \]
    with $\pi$ equal to the identity map.
    In this case, $\FHh^w(\Omega, T)/\log 2$ 
    is equal to the Hausdorff dimension of $\Omega$, whereas
    $\h^w(\Omega, T)/\log 2$ and $\hinf^w(\Omega, T)/\log 2$ 
    are the upper and lower Minkowski dimensions of $\Omega$,
    respectively. 
    It is then elementary to construct a closed and noninvariant subset 
    $\Omega \subset X$ that satisfies 
    the strict inequalities
    \[  \FHh^w(\Omega, T) < \hinf^w(\Omega, T) < \h^w(\Omega, T). \]
     For this particular map $\pi$, 
     we always have $\FHh^w(\Omega, T) \leq  \hinf^w(\Omega, T)$.
     With a little more effort, however, 
     one can also construct an equivariant continuous map
     $\pi\colon (X, T)\to (Y, S)$ and a closed and noninvariant 
     subset $\Omega \subset X$ such that, for some weight $w\in (0,1)$,
     \begin{equation} \label{eq: lower version is irrelevant, introduction}
     \hinf^w(\Omega, T)  < \FHh^w(\Omega, T) < \h^w(\Omega, T).  
     \end{equation}
     In particular, we cannot replace $\h^w(\Omega, T)$ by $\hinf^w(\Omega, T)$
     in the statement of Theorem \ref{theorem: main theorem}.
     Therefore, the limsup version $\h^w(\Omega, T)$, 
     rather than $\hinf^w(\Omega, T)$, 
     is the relevant quantity in the context of the present paper.
     The example satisfying \eqref{eq: lower version is irrelevant, introduction}
     is constructed in \S \ref{section: example}.
  \end{enumerate}
\end{remark}

\subsection{Relative weighted variational principles} 
\label{subsection: relative weighted variational principles}

The standard variational principle for topological entropy has
a relative version \cite{Ledrappier--Walters}.
Similarly, the two weighted variational principles
\eqref{eq: Feng--Huang variational principle} and
\eqref{eq: variational principle for covering weighted entropy}
reviewed in \S\ref{subsection: Background and main result}
have been extended to the relative setting.
In this subsection, we explain how
Theorem \ref{theorem: main theorem} clarifies the relationship
between these relative variational principles.
This application is a major motivation for the present work.

Let $(X,T)$, $(Y,S)$, and $(Z,R)$ be dynamical systems,
and let $\pi\colon X\to Y$, $\rho\colon X\to Z$, and
$\theta\colon Y\to Z$ be equivariant continuous maps satisfying
$\rho=\theta\circ\pi$.
Assume that $\rho$ is surjective.
\[
\begin{tikzcd}
(X,T) \arrow[r,"\pi"] \arrow[dr, two heads, "\rho"'] & (Y,S) \arrow[d,"\theta"] \\
& (Z,R)
\end{tikzcd}
\]
Let $\nu\in \mathscr{M}^R(Z)$ be an $R$-invariant Borel probability measure on $Z$.
In this setting, Yin \cite[Theorem 2.9]{Yin} proved 
\begin{equation} \label{eq: relative variational principle for covering weighted entropy}
  \begin{split}
   & \int_Z \h^w\left(\rho^{-1}(z),\pi, T\right) \, d\nu(z) \\
   & = \sup\left\{w h_\mu(T|R) + (1-w) h_{\pi_*\mu}(S|R) \middle|\, 
    \mu \in \mathscr{M}^T(X)  \text{ with }\rho_*\mu = \nu\right\}.
   \end{split}
\end{equation}
Under the same assumptions, Alibabaei and the author
\cite{Alibabaei--Tsukamoto} proved
\begin{equation} 
   \label{eq: relative variational principle for Feng--Huang weighted entropy}
  \begin{split}
   & \int_Z \FHh^w\left(\rho^{-1}(z),\pi, T\right) \, d\nu(z) \\
   & = \sup\left\{w h_\mu(T|R) + (1-w) h_{\pi_*\mu}(S|R) \middle|\, 
    \mu \in \mathscr{M}^T(X)  \text{ with }\rho_*\mu = \nu\right\}.
   \end{split}
\end{equation}    
Here $\h^w\left(\rho^{-1}(z),\pi, T\right)$ and $\FHh^w\left(\rho^{-1}(z),\pi, T\right)$ 
are the covering and Feng--Huang 
$w$-weighted topological entropies of the set $\rho^{-1}(z)\subset X$ 
with respect to the map $\pi\colon (X,T) \to (Y,S)$, respectively.
The quantities $h_\mu(T|R)$ and $h_{\pi_*\mu}(S|R)$ 
denote the conditional Kolmogorov--Sinai entropies.

The two relative variational principles 
\eqref{eq: relative variational principle for covering weighted entropy} and
\eqref{eq: relative variational principle for Feng--Huang weighted entropy}
have the same right-hand side.
Applying them on ergodic components yields the following
almost-everywhere identity; see
\cite[Corollary 2.3]{Alibabaei--Tsukamoto}:
\begin{equation} \label{eq: equivarence of relative versions}
  \FHh^w\left(\rho^{-1}(z),\pi, T\right)  = \h^w\left(\rho^{-1}(z),\pi, T\right) \quad \text{for $\nu$-a.e. } z\in Z. 
\end{equation}  
This is the relative version of the identity \eqref{eq: coincidence of two weighted topological entropy}
reviewed in \S \ref{subsection: Background and main result}.
Notice that $\rho^{-1}(z)$ is not $T$-invariant in general. 
Therefore \eqref{eq: equivarence of relative versions} 
does not follow from \eqref{eq: coincidence of two weighted topological entropy}.

The relative variational principles 
\eqref{eq: relative variational principle for covering weighted entropy} and
\eqref{eq: relative variational principle for Feng--Huang weighted entropy}
are not only of theoretical interest but also have a concrete application to fractal geometry.
Using the identity \eqref{eq: equivarence of relative versions}, 
we can compute the Hausdorff dimension of 
the intersection of two randomly translated Bedford--McMullen carpets; 
see \cite[Theorem 2.4]{Alibabaei--Tsukamoto}.
This extends the pioneering work of Kenyon and Peres \cite{Kenyon--Peres_intersection} 
from the self-similar setting to the self-affine setting.

Although the relative versions 
\eqref{eq: relative variational principle for covering weighted entropy} and
\eqref{eq: relative variational principle for Feng--Huang weighted entropy} 
are very interesting, 
their proofs are rather lengthy and are based on substantially different ideas.
The application to fractal geometry requires both versions; neither
one alone is sufficient.
To understand the existing derivation in full, one must therefore
work through both proofs.

The present work sheds new light on this issue.
By Theorem \ref{theorem: main theorem}, we have
\begin{equation}
\label{eq: comparison of two fiber entropies}
\FHh^w\left(\rho^{-1}(z),\pi,T\right)
\leq
\h^w\left(\rho^{-1}(z),\pi,T\right)
\qquad
\text{for every $z\in Z$}.
\end{equation}
Therefore, in order to establish both
\eqref{eq: relative variational principle for covering weighted entropy} and
\eqref{eq: relative variational principle for Feng--Huang weighted entropy},
it is enough to prove the upper bound
\begin{equation*}
\begin{split}
&\int_Z
\h^w\left(\rho^{-1}(z),\pi,T\right)\,d\nu(z)\\
&\leq
\sup\left\{
w h_\mu(T\mid R)
+
(1-w)h_{\pi_*\mu}(S\mid R)
\ \middle|\
\mu\in\mathscr{M}^T(X),\ \rho_*\mu=\nu
\right\},
\end{split}
\end{equation*}
and the lower bound
\begin{equation*}
\begin{split}
&\int_Z
\FHh^w\left(\rho^{-1}(z),\pi,T\right)\,d\nu(z)\\
&\geq
\sup\left\{
w h_\mu(T\mid R)
+
(1-w)h_{\pi_*\mu}(S\mid R)
\ \middle|\
\mu\in\mathscr{M}^T(X),\ \rho_*\mu=\nu
\right\}.
\end{split}
\end{equation*}
Indeed, combining these two inequalities with
\eqref{eq: comparison of two fiber entropies}, we obtain
\begin{equation*}
  \begin{split}
  &\sup\left\{
w h_\mu(T\mid R)
+
(1-w)h_{\pi_*\mu}(S\mid R)
\ \middle|\
\mu\in\mathscr{M}^T(X),\ \rho_*\mu=\nu
\right\} \\
&\leq 
\int_Z
\FHh^w\left(\rho^{-1}(z),\pi,T\right)\,d\nu(z) \\
&\leq
\int_Z
\h^w\left(\rho^{-1}(z),\pi,T\right)\,d\nu(z) \\
&\leq \sup\left\{
w h_\mu(T\mid R)
+
(1-w)h_{\pi_*\mu}(S\mid R)
\ \middle|\
\mu\in\mathscr{M}^T(X),\ \rho_*\mu=\nu
\right\},
  \end{split}
\end{equation*}  
and hence all the quantities are equal to each other.

Thus, instead of proving both directions of each variational
principle independently,
it suffices to prove one inequality for each of them.
The remaining two inequalities are supplied simultaneously 
by Theorem \ref{theorem: main theorem}.
In other words, Theorem \ref{theorem: main theorem} 
allows us to replace \lq\lq{}half\rq\rq{} 
of the proofs of the two relative variational principles.
Notice that, for this application, it is crucial that Theorem \ref{theorem: main theorem}
deals with noninvariant subsets $\Omega\subset X$ because the fibers
$\rho^{-1}(z)$ are not $T$-invariant in general.

We do not claim that this approach yields better proofs of the
two relative variational principles.
The original proofs have their own advantages, and our proof of
Theorem \ref{theorem: main theorem} 
is itself nontrivial.\footnote{Indeed,
our proof is strongly influenced by Feng and Huang's proof of
their variational principle \eqref{eq: Feng--Huang variational principle}.}
Nevertheless, we believe that this approach provides a new
perspective on these foundational results and helps clarify
their underlying structure.

\subsection{Use of generative AI}
This paper grew out of intensive use of ChatGPT (GPT-5.6 and 6).
Almost all the ideas underlying the proofs were generated by ChatGPT.
In particular, the ideas of using a flow construction in
\S\ref{section: proof of the main theorem for Omega = X}
and applying linear-programming duality in
\S\ref{section: proof of the general case of the main theorem}
were generated solely by ChatGPT.

The author's role was to formulate a sequence of questions,
examine the resulting ideas, and streamline the arguments.
All AI-generated suggestions and outputs were checked,
verified, and edited by the author.
The author takes full responsibility for all mathematical statements
and proofs in this paper.

\section{Proof of Theorem \ref{theorem: main theorem} for $\Omega = X$}
\label{section: proof of the main theorem for Omega = X}

In this section, we prove a special case of Theorem \ref{theorem: main theorem}, 
namely the case $\Omega=X$.
The general case will be treated in 
\S\ref{section: proof of the general case of the main theorem}.
The proof of the general case does not use the argument in this section,
so readers may proceed directly to that proof.
Nevertheless, 
the argument for $\Omega=X$ is substantially simpler 
and helps clarify the essential difficulties that arise in the general setting.
For this reason, we present this special case separately.

Let $\pi\colon (X, T)\to (Y, S)$ be an equivariant continuous map between dynamical systems.
Let $0\leq w \leq 1$. We will prove 
\[  \FHh^w(X, T) \leq \h^w(X, T). \]
We can assume that $\pi$ is surjective 
because otherwise we can replace $Y$ by $\pi(X)$ and $S$ by the restriction of $S$ to $\pi(X)$.
We take metrics $\mathbf{d}$ and $\mathbf{d}^\prime$ on $X$ and $Y$, respectively.

We first consider the endpoint case $w=0$. 
In this case, we have 
\[ B^0_n(x, \varepsilon) 
= \pi^{-1}\{y \in Y : \mathbf{d}^\prime_n(y, \pi(x)) < \varepsilon\}, \]
and $\h^0(X, T) = h_{\text{top}}\left(Y, S\right)$ 
is the topological entropy of $(Y, S)$.
A cover of $Y$ by sets of $\mathbf d'_n$-diameter less than
$\varepsilon$ gives a cover of $X$ by the same number of
$0$-weighted Bowen balls of length $n$ and radius $\varepsilon$.
Therefore
\[
\FHh^0(X,T)\leq\h^0(X,T).
\]
Hence, in the remainder of this section, we assume $0 <w \leq 1$.

The desired inequality is immediate if $\h^w(X,T)=\infty$.
We therefore assume that $\h^w(X,T)<\infty$.
Fix $s>\h^w(X,T)$. It suffices to prove that
$\FHh^w(X,T)\leq s$.

Take any $N_0 \in \mathbb{N}$ and $\varepsilon \in (0,1)$.
Since $\h^w(X, T) < s$, we can choose $N\geq N_0$ and open covers 
\begin{itemize}
  \item $Y= V_1\cup \dots \cup V_a$ with 
  $\diam(V_i, \mathbf{d}^\prime_N) < \varepsilon$ for all $1\leq i \leq a$, 
  \item $\pi^{-1}(V_i) = U_{i 1}\cup U_{i 2} \cup \dots \cup U_{i t_i}$ 
  with $\diam(U_{ij}, \mathbf{d}_N) < \varepsilon$,
\end{itemize}
such that 
\[ Z := \sum_{i=1}^a t_i^w < e^{sN}. \]
We discard any empty $V_i$.
Since $\pi$ is surjective, we then have $t_i\geq1$ for every $i$.
From these data, we will construct a rooted tree and a \lq\lq{}flow\rq\rq{} on it.
The construction is reminiscent of the flow-based proof of
Frostman's lemma in \cite[\S 3.1]{Bishop--Peres}\footnote{Frostman's lemma is a 
classical result in geometric measure theory 
that provides a criterion for the existence of measures with prescribed dimensional properties. 
Therefore, although the present proof is logically \lq\lq{}measure-free\rq\rq{},
it is inspired by a measure-theoretic argument.}.

We define a rooted tree $\Gamma$ as follows.
 \begin{itemize}
   \item The root of $\Gamma$ is denoted by $(\emptyset)$ and is regarded as the unique vertex of level $0$.
   For each integer $q\geq 1$, a vertex of level $q$ is a tuple of the form
   \[ v= (i_1,\dots,i_q;j_1,\dots,j_{\lceil wq\rceil}), \]
   where $i_1,\dots,i_q\in \{1,\dots,a\}$ and $j_k\in \{1,\dots,t_{i_k}\}$ ($1\leq k\leq\lceil wq\rceil$).  
   The vertex set of $\Gamma$ consists of the root $(\emptyset)$ together with all vertices of levels $q\geq1$.
   \item For each vertex $v = (i_1,\dots,i_q;j_1,\dots,j_{\lceil wq\rceil})$ of level $q\geq 1$, we define its \textit{parent} by 
   \[ p(v) =  (i_1,\dots,i_{q-1};j_1,\dots,j_{\lceil w(q-1)\rceil}). \]
   This is a vertex of level $q-1$. If $q=1$, we set $p(v) = (\emptyset)$.
   We join $p(v)$ to $v$ by an edge $p(v)\rightarrow v$. 
 \end{itemize}
 Notice that, in passing from a vertex $v$ to its parent $p(v)$, 
 the $i$-word always loses its last symbol, 
 whereas the $j$-word loses its last symbol only when
$\lceil w(q-1)\rceil=\lceil wq\rceil-1$.
Thus, if $\lceil w(q-1)\rceil=\lceil wq\rceil$, the $j$-word remains unchanged.

A \textbf{flow} on $\Gamma$ is a nonnegative function $f$ 
defined on the set of edges that satisfies 
\[ f(p(v) \rightarrow v) = \sum_{u: p(u) = v} f(v\rightarrow u) \]
for all vertices $v$ of level $q\geq 1$. 
Here, the sum on the right-hand side is taken over 
all vertices $u$ of level $q+1$ with $p(u) = v$
(namely, all children $u$ of $v$).
One may visualize $\Gamma$ as a network of pipes through which water flows. 
We imagine that water is supplied from outside at the root 
and then flows down along the edges of the tree. 
The quantity $f(p(v)\to v)$ represents the amount of water flowing into $v$, 
while $f(v\to u)$ represents the amount flowing from $v$ to its child $u$. 
Thus, the above defining identity for a flow simply expresses that, 
at each vertex of level $q\geq 1$, 
the total amount of incoming water is equal to the total amount of outgoing water.

For a flow $f$ on $\Gamma$, we set 
\[ \norm{f} = \sum_{v:\text{ level $1$}} f\left((\emptyset)\rightarrow v\right). \]
Here the sum is taken over all vertices $v$ of level $1$.
Thus, $\norm{f}$ represents the total amount of water 
supplied from outside at the root.
We call $f$ a \textbf{unit flow} if $\norm{f} = 1$.

For a flow $f$ and a vertex $v$ of level $q\geq 1$, we set 
\[ f_{\mathrm{in}}(v) = f\left(p(v)\rightarrow v\right). \]
This represents the amount of water flowing into the vertex $v$.

Now we define a unit flow $f$ on $\Gamma$ inductively with respect to the level:
\begin{enumerate}
  \item[(1)] Let $v = (i; j)$ be a vertex of level $1$. We define 
  \[ f\left((\emptyset)\rightarrow v\right) = \frac{t_i^{w-1}}{Z}. \]
  By the definition of $Z$, we have $\norm{f} = 1$. 
  So the flow $f$ constructed below will be a unit flow.
  \item[(2)] Let $v = (i_1, \dots, i_q; j_1, \dots, j_{\lceil wq \rceil})$ 
  be a vertex of level $q\geq 2$. 
  Suppose that the flow has already been defined up to level $q-1$.
  If $\lceil w(q-1)\rceil = \lceil w q\rceil$, then we define 
  \[ f\left(p(v)\rightarrow v\right) 
  = \frac{t_{i_q}^w}{Z}\cdot f_{\mathrm{in}}\left(p(v)\right). \]
  If $\lceil w(q-1)\rceil = \lceil w q\rceil-1$, then we define
  \[ f\left(p(v)\rightarrow v\right) 
  = \frac{t_{i_q}^w}{t_{i_{\lceil wq \rceil}} Z} 
  \cdot f_{\mathrm{in}}\left(p(v)\right). \]
\end{enumerate}

It is straightforward to 
check that the above inductive construction 
satisfies the flow conservation condition at every vertex of level $q\geq 1$.
Hence it defines a unit flow $f$ on $\Gamma$.
Explicitly, for $v = (i_1, \dots, i_q; j_1, \dots, j_{\lceil wq \rceil})$
\[ f_{\mathrm{in}}(v) 
= \prod_{k=1}^q \frac{t_{i_k}^w}{Z} \cdot \prod_{k=1}^{\lceil wq \rceil} \frac{1}{t_{i_k}}. \]
Set
\[ g(v) =  
w\left(\sum_{k=1}^q \log t_{i_k}\right) - \left(\sum_{k=1}^{\lceil wq \rceil} \log t_{i_k}\right). \]
Then
\begin{equation} \label{eq: incoming flow formula}
  \log f_{\mathrm{in}}(v) = -q \log Z + g(v).
\end{equation}

We consider an infinite path starting at the root and going down the tree:
\[ \gamma\colon (\emptyset)\to v_1\to v_2\to\cdots, \]
where $v_q=(i_1,\dots,i_q;j_1,\dots,j_{\lceil wq\rceil})$ is a vertex of level $q$.
Such an infinite path is determined 
by the two sequences $i_1,i_2,\dots$ and $j_1,j_2,\dots$ that satisfy 
\[ i_k \in \{1,2,\dots, a\}, \qquad j_k \in \{1,2,\dots, t_{i_k}\}. \]
Therefore we also denote the path by
$\gamma=(i_1,i_2,\dots;j_1,j_2,\dots)$.

The next lemma is a key technical ingredient.

\begin{lemma} \label{lemma: infinite path}
Suppose that we are given an infinite path 
$\gamma = (i_1, i_2, \dots ; j_1, j_2, \dots)$ starting at the root, and
let $v_q = (i_1, \dots, i_q; j_1, \dots, j_{\lceil wq \rceil})$.
For every $\delta>0$, there exist infinitely many $q$ such that
\[ g(v_q) > -\delta q. \]
\end{lemma}

\begin{proof}
We set $M = \max_{1\leq i \leq a} \log t_i$ and 
\[ a_q = \sum_{k=1}^q \log t_{i_k}, \quad (q\geq 1). \]
We have $g(v_q) = w a_q - a_{\lceil wq\rceil}$.

Suppose the conclusion is false.
Then there exists $q_0>0$ such that we have $g(v_q) \leq -\delta q$ for all $q\geq q_0$.
Equivalently
\[ a_{\lceil w q \rceil} \geq  w a_q + \delta q, \quad (q\geq q_0). \]
Set $b_q = a_q/q$. We have $0\leq b_q \leq M$.
For $q\geq q_0$
  \[
    b_{\lceil wq\rceil}  \geq  \frac{wq}{\lceil wq \rceil} b_q + \frac{\delta q}{\lceil wq\rceil} 
    = b_q - \left(1-\frac{wq}{\lceil wq\rceil}\right) b_q + \frac{\delta q}{\lceil wq\rceil}.
   \]
Since $1-\frac{wq}{\lceil wq\rceil} < \frac{1}{\lceil wq\rceil}$ and $\frac{q}{\lceil wq\rceil}\geq 1$, 
\[ b_{\lceil wq\rceil} \geq b_q  - \frac{M}{\lceil wq\rceil} + \delta. \] 
Then, for $q \geq \max\left(q_0, \frac{2M}{w\delta}\right)$
\[  b_{\lceil wq\rceil} \geq  b_q + \frac{\delta}{2}. \]

Since $\lceil wq\rceil\to\infty$ as $q\to\infty$,
taking upper limits gives
\[
\limsup_{q\to\infty}b_q
\geq
\limsup_{q\to\infty}b_{\lceil wq\rceil}
\geq
\limsup_{q\to\infty}b_q+\frac{\delta}{2}.
\]
This is impossible because $0\leq b_q\leq M$.
\end{proof}

Fix an arbitrary positive number $\delta$.
For each infinite path $\gamma = (i_1, i_2, \dots ; j_1, j_2, \dots)$, let $v_q = (i_1, \dots, i_q; j_1, \dots, j_{\lceil wq \rceil})$
and choose the smallest integer $q = q(\gamma)\geq 1 $ such that $g(v_q) > -\delta q$.
We define $\mathcal{C}$ as the set of vertices $v$ for which there exists an infinite path $\gamma = (i_1, i_2, \dots ; j_1, j_2, \dots)$
that satisfies $v = (i_1, \dots, i_{q(\gamma)}; j_1, \dots, j_{\lceil w q(\gamma) \rceil})$.

The integers $q(\gamma)$ are uniformly bounded over all infinite paths $\gamma$.
Indeed, if they were not uniformly bounded, a standard diagonal argument would produce an infinite path
$\gamma$ contradicting Lemma \ref{lemma: infinite path}.
It follows that $\mathcal{C}$ is a finite set.

The set $\mathcal{C}$ is a \textbf{cut-set}: every infinite path starting at the root meets $\mathcal{C}$.
Moreover, $\mathcal{C}$ is an \textbf{antichain}\footnote{In particular, $\mathcal{C}$ is minimal 
in the sense that no proper subset of $\mathcal{C}$ is a cut-set.}: 
no two distinct vertices of $\mathcal{C}$ lie on the same infinite path. 
This follows immediately from the definition of the first stopping time $q(\gamma)$.

\begin{claim} \label{claim: minimal cut-set}
 \[ \sum_{v\in \mathcal{C}} f_{\mathrm{in}}(v) = 1. \]
\end{claim}

\begin{proof}
Intuitively, this is an immediate consequence of flow conservation, but we provide the details for completeness.
Choose an integer $L$ so large that $L>q(\gamma)$ for all infinite paths $\gamma$.
Let $\mathcal{V}_L$ denote the set of all vertices of level $L$.
For each $v\in \mathcal{C}$, let $\mathcal{V}_L(v)$ be the set of all vertices $u\in\mathcal{V}_L$
for which $v$ is an ancestor of $u$; equivalently, $v$ lies on the path from the root to $u$.

Repeated application of the flow conservation rule gives
\[ f_{\mathrm{in}}(v) = \sum_{u\in \mathcal{V}_L(v)} f_{\mathrm{in}}(u).
\]
Since $\mathcal{C}$ is a cut-set, every vertex $u$ of level $L$ has an
ancestor belonging to $\mathcal{C}$. 
This ancestor is unique because $\mathcal{C}$ is an antichain. 
Therefore, the sets $\mathcal{V}_L(v)$ ($v\in\mathcal{C}$) form a partition of $\mathcal{V}_L$. 
It follows that
\[
\sum_{v\in\mathcal{C}} f_{\mathrm{in}}(v)
=
\sum_{v\in\mathcal{C}}\sum_{u\in \mathcal{V}_L(v)}f_{\mathrm{in}}(u)
=
\sum_{\substack{u\in \mathcal{V}_L}}f_{\mathrm{in}}(u).
\]
The flow conservation rule also implies that the total incoming flow
is the same at every level. Hence
\[ \sum_{\substack{u\in \mathcal{V}_L}}f_{\mathrm{in}}(u) = \sum_{u\in \mathcal{V}_1} f_{\mathrm{in}}(u) = \norm{f} = 1. \]
\end{proof}

We now return to the proof of $\FHh^w(X, T) \leq s$.
For each $v = (i_1, \dots, i_q; j_1, \dots, j_{\lceil wq \rceil}) \in \mathcal{C}$, 
we define an open subset $E_v$ of $X$ by
\[ E_v = 
\bigcap_{k=1}^{q} T^{-(k-1)N} \pi^{-1}(V_{i_k}) 
\cap \bigcap_{k=1}^{\lceil wq\rceil} T^{-(k-1)N} U_{i_k j_k}. \]

\begin{claim} \label{claim: diameter of E_v}
    Let $v\in \mathcal{C}$ be a vertex of level $q$. 
     Then the diameter of $E_v$ with respect to $\mathbf{d}^w_{qN}$ is smaller than $\varepsilon$.
\end{claim}

\begin{proof}
Let $v = (i_1, \dots, i_q; j_1, \dots, j_{\lceil wq \rceil}) \in \mathcal{C}$, and let $x, y\in E_v$.
Then 
\begin{align*}
   & S^{(k-1)N}\pi(x), S^{(k-1)N} \pi(y)  \in V_{i_k} \quad (1\leq k \leq q), \\
   & T^{(k-1)N} x, T^{(k-1)N} y  \in U_{i_k j_k} \quad (1\leq k \leq \lceil wq \rceil). 
\end{align*}   
It follows that 
\begin{align*}
  & \mathbf{d}^\prime_N\left(S^{(k-1)N}\pi(x), S^{(k-1)N} \pi(y)\right) 
  \leq \diam(V_{i_k}, \mathbf{d}^\prime_N) \quad (1\leq k \leq q), \\
  & \mathbf{d}_N\left(T^{(k-1)N}x, T^{(k-1)N}y\right) 
  \leq \diam(U_{i_k j_k}, \mathbf{d}_N)  
  \quad (1\leq k \leq \lceil wq \rceil). 
\end{align*} 
From the former inequality, 
we obtain 
\[ \mathbf{d}^\prime_{qN}\left(\pi(x), \pi(y)\right) 
\leq \max_{1\leq k \leq q} \diam(V_{i_k}, \mathbf{d}^\prime_N). \]
Since $\lceil wq \rceil N \geq \lceil wq N\rceil$, 
from the latter inequality, we also obtain 
\[ \mathbf{d}_{\lceil wq N\rceil}(x,y) 
\leq \max_{1\leq k \leq \lceil wq \rceil}
 \diam(U_{i_k j_k}, \mathbf{d}_N). \]
Therefore, 
\[ \mathbf{d}^w_{qN}(x,y) \leq 
\max\left\{\max_{1\leq k \leq q} \diam(V_{i_k}, \mathbf{d}^\prime_N),
\max_{1\leq k \leq \lceil wq \rceil} \diam(U_{i_k j_k}, \mathbf{d}_N)\right\}
< \varepsilon. \]
\end{proof}

\begin{claim} \label{claim: covering by a cut-set}
\[ X = \bigcup_{v\in \mathcal{C}} E_v. \]
\end{claim}

\begin{proof}
Let $x\in X$ be an arbitrary point.
For each $k\geq 1$, take $i_k$ and $j_k$ so that 
$S^{(k-1)N}\pi(x) \in V_{i_k}$ and $T^{(k-1)N} x\in U_{i_k j_k}$.
These choices determine an infinite path
$\gamma = (i_1, i_2, \dots; j_1, j_2, \dots)$.
Let $q = q(\gamma)$. We have $v_q\in \mathcal{C}$ and hence $x\in E_{v_q}$.
\end{proof}

For each $v\in \mathcal{C}$ with $E_v\neq \emptyset$, we choose a point $x_v\in E_v$.
Let $q_v\geq 1$ be the level of $v\in \mathcal{C}$. 
It follows from Claims \ref{claim: diameter of E_v} and \ref{claim: covering by a cut-set} that
\[ X = \bigcup_{v\in \mathcal{C}: E_v\neq \emptyset} B^w_{q_v N}(x_v, \varepsilon), \qquad 
   q_v N \geq N \geq N_0. \]
For each $v\in \mathcal{C}$, by \eqref{eq: incoming flow formula} and the definition of $\mathcal{C}$
\[ \log f_{\mathrm{in}}(v) = -q_v \log Z + g(v) > -q_v \log Z - \delta q_v. \]
Equivalently, $\exp\left(-q_v(\log Z + \delta)\right) < f_{\mathrm{in}}(v)$. Hence 
\[ \exp\left(-q_v N\left(\frac{\log Z}{N} + \delta \right)\right) \leq \exp\left(-q_v(\log Z + \delta)\right) < f_{\mathrm{in}}(v). \]
By Claim \ref{claim: minimal cut-set}
\[ \sum_{v\in \mathcal{C}} 
   \exp\left(-q_v N\left(\frac{\log Z}{N} + \delta\right)\right) 
   < \sum_{v\in \mathcal{C}} f_{\mathrm{in}}(v) = 1. \]
Since $Z < e^{sN}$, 
\[ \sum_{v\in \mathcal{C}} e^{-q_v N (s+\delta)} < 1. \]
Therefore 
\[ \Lambda^{w, s + \delta}_{N_0, \varepsilon}(X) \leq 1. \]
This bound holds for every $N_0$, with $s$, $\delta$, and
$\varepsilon$ fixed. Letting $N_0\to\infty$ gives
\[
\Lambda^{w,s+\delta}_{\varepsilon}(X)\leq1.
\]
Thus $\FHh^w(X, T, \varepsilon) \leq s + \delta$.
Letting $\delta\to 0$ and $\varepsilon \to 0$, we obtain $\FHh^w(X, T) \leq s$.
Since $s$ is an arbitrary positive number with $s>\h^w(X, T)$, this proves 
\[ \FHh^w(X, T) \leq \h^w(X, T). \]

We conclude this section with a brief overview of the proof.
The main task is to turn a two-step cover of $X$ 
into an efficient
cover by weighted Bowen balls.
For this purpose, we introduce an auxiliary rooted tree and
a unit flow on it.
A suitably chosen finite cut-set yields the desired cover:
the cut-set property ensures 
that the associated balls cover $X$,
while the flow estimates control their total covering cost.

\section{Proof of the general case of Theorem \ref{theorem: main theorem}}
\label{section: proof of the general case of the main theorem}

In this section we prove Theorem \ref{theorem: main theorem} 
in the general case of an arbitrary subset $\Omega\subset X$.
First we explain what makes the general case more difficult 
than the special case $\Omega=X$.

The argument of the preceding section starts from a fixed reference
time $N$ and proceeds through the time scales $N,2N,3N,\dots$,
applying the same two-step cover at successive points $T^{kN}x$
along each orbit.
For an arbitrary subset $\Omega\subset X$, however,
$x\in\Omega$ does not imply $T^{kN}x\in\Omega$.
Thus, a cover chosen to estimate $\#^w(\Omega,N,\varepsilon)$
cannot in general be used repeatedly in this way.

The key idea of the present section is to reverse the direction
in which the time scale changes.
Assume $0<w<1$. (The endpoint cases $w=0, 1$ are simpler.)
Starting from a reference time $N$, we pass successively to
shorter time scales,
\[
N,\ \lceil wN\rceil,\
\left\lceil w\lceil wN\rceil\right\rceil,\ \dots,
\]
rather than extending the time horizon beyond $N$.
This avoids the need to apply the same cover at later points
of the orbit, which may lie outside $\Omega$.

The scale-descent inequality in
\S\ref{subsection: scale-descent inequality}
provides an abstract formulation of this idea.
We apply it in Step~4 of
\S\ref{subsection: proof of main theorem},
where we select a suitable time scale 
from the descending sequence.

This section is organized as follows.
In \S\ref{subsection: scale-descent inequality}, we establish
a scale-descent inequality that converts an estimate involving
two time scales into a weighted estimate at a single scale.
In \S\ref{subsection: fractional covering problem and its dual},
we recall the duality between fractional covering and fractional
packing problems.
We use this duality in
\S\ref{subsection: finitary dynamical Frostman lemma}
to prove a finitary dynamical Frostman lemma, which provides
finitely supported probability measures satisfying Frostman-type
bounds over finite time ranges.
In \S\ref{subsection: partition construction and entropy estimates},
we construct finite families of partitions with controlled
exceptional sets and derive entropy estimates to control
boundary errors.
Finally, in \S\ref{subsection: proof of main theorem},
we combine these ingredients to prove
Theorem \ref{theorem: main theorem}.

\subsection{Scale-descent inequality}
\label{subsection: scale-descent inequality}

In the proof of Theorem \ref{theorem: main theorem}, 
we will encounter the following situation.
Suppose that two nonnegative sequences $(a_n)$ and $(b_n)$ satisfy
\[
a_n+b_{\lceil wn\rceil}\geq sn
\]
for some $s\geq0$.
Assuming also that $b_n$ grows at most linearly, 
we would like to find an index $n'$ for which
\[
\frac{a_{n'}+wb_{n'}}{n'}
\geq s-(\text{small error term}).
\]
The following lemma gives a quantitative estimate of this form.

\begin{lemma}[Scale-descent inequality]
\label{lemma: scale-descent inequality}
Let $0<w\leq1$, $s\geq0$, and $C\geq0$.
Let $K$ and $L$ be positive integers.
Suppose that two nonnegative finite sequences
$(a_n)_{1 \leq n\leq L}$ and $(b_n)_{1\leq n\leq L}$
satisfy
\[
b_n\leq Cn,  \qquad
a_n+b_{\lceil wn\rceil}\geq sn
\qquad
(1  \leq n\leq L).
\]
Define integers $n_0,n_1,\dots,n_K$ recursively by
\[
n_0=L,
\qquad
n_{k+1}=\lceil wn_{k}\rceil
\quad (0\leq k <K).
\]
Then there exists an integer $k$ with $0\leq k<K$ such that
\[
\frac{a_{n_k}+wb_{n_k}}{n_k}
\geq
s-\frac{wC}{K}-\frac{C}{w^K L}.
\]
\end{lemma}

\begin{proof}
Set $r_k = b_{n_k}/n_k$ for $0\leq k\leq K$.
Then $0\leq r_k\leq C$ for all $k$.
We also set $\theta_k = n_{k+1}-wn_k$ for $0\leq k<K$.
We have $0\leq \theta_k<1$ and $n_{k+1} = wn_k + \theta_k$. 
From the assumption $a_{n_k}+b_{n_{k+1}}\geq sn_k$
$(0\leq k <K)$, we obtain
\[ s - \frac{a_{n_k}+wb_{n_k}}{n_k} 
   \leq \frac{a_{n_k}+b_{n_{k+1}}}{n_k} - \frac{a_{n_k}+wb_{n_k}}{n_k}
   = \frac{b_{n_{k+1}}-wb_{n_k}}{n_k}. \]
We have 
\begin{align*}
   b_{n_{k+1}}-wb_{n_k} 
&= n_{k+1}r_{k+1}-wn_kr_k  \\
&= (wn_k+\theta_k)r_{k+1}-wn_kr_k \\
&= wn_k(r_{k+1}-r_k)+\theta_kr_{k+1}. 
\end{align*}
Hence 
\[ s - \frac{a_{n_k}+wb_{n_k}}{n_k} \leq
  w(r_{k+1}-r_k)+\frac{\theta_k}{n_k}r_{k+1} 
  \leq w(r_{k+1}-r_k)+\frac{C}{n_k}. \]
In the second inequality, 
we used the bounds $0\leq \theta_k<1$ and $0\leq r_{k+1}\leq C$.

Summing over $k=0,\dots,K-1$ and using telescoping cancellation,
we obtain
\[
\sum_{k=0}^{K-1} \left(s - \frac{a_{n_k}+wb_{n_k}}{n_k}\right)
\leq w(r_K-r_0)+\sum_{k=0}^{K-1}\frac{C}{n_k}
\leq wC+\sum_{k=0}^{K-1}\frac{C}{n_k}.
\]
Therefore, there exists an integer $k$ with $0\leq k<K$ such that
\[ s - \frac{a_{n_k}+wb_{n_k}}{n_k}
\leq \frac{1}{K}\left(wC+\sum_{k=0}^{K-1}\frac{C}{n_k}\right)
= \frac{wC}{K}+\frac{C}{K}\sum_{k=0}^{K-1}\frac{1}{n_k}. \]
Since $n_k \geq w^k L \geq w^K L$, we conclude that
\[ s - \frac{a_{n_k}+wb_{n_k}}{n_k}
\leq \frac{wC}{K}+\frac{C}{w^K L}. \]
\end{proof}

For fixed $w$ and $C$, 
the error terms in the lemma can be made arbitrarily
small by first choosing $K$ sufficiently large and then choosing $L$
sufficiently large.
Moreover, since $n_k\geq w^K L$, 
the selected time scale $n_k$ can also be kept
arbitrarily large.

We also note that the above proof does not require the sequences $(a_n)$
and $(b_n)$ to be defined for all $1\leq n\leq L$.
In fact, it only uses the values
\[
a_{n_0},\dots,a_{n_{K-1}}
\qquad\text{and}\qquad
b_{n_0},\dots,b_{n_K}.
\]
Since $n_k\geq\lceil w^kL\rceil$, it suffices for the nonnegative
sequences $(a_n)$ and $(b_n)$ to be defined on the ranges
\[
\lceil w^{K-1}L\rceil\leq n\leq L
\qquad\text{and}\qquad
\lceil w^KL\rceil\leq n\leq L,
\]
respectively.
The bound $b_n\leq Cn$ need only hold on the latter range,
and the inequality $a_n+b_{\lceil wn\rceil}\geq sn$
need only hold on the former.
The conclusion of the lemma remains unchanged.

This observation will be useful in the proof of
Theorem \ref{theorem: main theorem}; see Step~4 of 
\S \ref{subsection: proof of main theorem}.
For simplicity, however, we have stated the lemma with both
sequences defined on the full interval $1\leq n\leq L$.

\subsection{Fractional covering problem and its dual}
\label{subsection: fractional covering problem and its dual}

In Section \ref{section: proof of the main theorem for Omega = X},
the key combinatorial ingredient was a flow argument on a rooted tree.
That argument does not extend directly to arbitrary subsets
$\Omega\subset X$.
Instead, we use a different combinatorial tool, 
namely the duality between
fractional covering and fractional packing problems.
This duality is well-known 
(see e.g. \cite[Chapter 1]{Williamson--Shmoys}, \cite[\S 10.1]{Matousek}) and
is a standard application of linear programming.
For readers less familiar with this combinatorial formulation, we review
it here.

We begin by recalling the duality theorem for linear programming.
For vectors $x=(x_1,\dots,x_m)^T$ and $y=(y_1,\dots,y_m)^T$
in $\mathbb{R}^m$, we write $x\geq y$ if $x_i\geq y_i$ for every
$1\leq i\leq m$.
All vector inequalities in this subsection 
are interpreted in this coordinatewise sense.
Let $r$ and $m$ be positive integers, let $A$ be a real $r\times m$
matrix, and let $b\in\mathbb{R}^r$ and $c\in\mathbb{R}^m$.
Define the feasible sets
\[
P:=\{x\in\mathbb{R}^m\mid x\geq0,\ Ax\geq b\},
\qquad
D:=\{y\in\mathbb{R}^r\mid y\geq0,\ A^Ty\leq c\}.
\]

The following is a fundamental result in linear programming
(see e.g. \cite[Proposition 10.1.2]{Matousek}).
\begin{theorem}[Linear programming duality] 
  \label{theorem: linear programming duality}
Suppose that 
$P\neq \emptyset$ and $\inf\{c^T \cdot x\mid x\in P\} >- \infty$.
Then we have $D\neq \emptyset$, and 
\[ \min_{x\in P} \left(c^T\cdot x\right) 
   = \max_{y\in D} \left(b^T\cdot y\right). \]
Here the minimum and maximum exist.
\end{theorem}

We now apply this theorem to the fractional covering problem.
Let $V = \{v_1, \dots, v_r\}$ be a finite set.
Let $\mathcal{E} = \{E_1, \dots, E_m\}$ be a family of 
subsets of $V$ that satisfies 
\[ V = \bigcup_{j=1}^m E_j. \]
Suppose we are given a positive number $a_j$ 
for each $j$ with $1\leq j\leq m$, which is interpreted as 
a \lq\lq{}cost of $E_j$\rq\rq{}. 

A collection of nonnegative numbers $\lambda_1, \dots, \lambda_m$ 
is called a \textbf{fractional covering of $V$}
with respect to $\mathcal{E}$ if for each $v_\alpha \in V$ we have 
\[ \sum_{j:\, v_\alpha \in E_j} \lambda_j \geq 1, \]
where the sum is taken over all indices 
$j$ with $v_\alpha \in E_j$.\footnote{If $\lambda_j\in\{0,1\}$ 
for every $j$, then the sets $E_j$ with $\lambda_j=1$ 
form an ordinary covering of $V$.
Thus, fractional covering is a relaxation of ordinary covering.}
The \textbf{cost} of a fractional covering is
$\sum_{j=1}^m a_j\lambda_j$.
Since $\mathcal{E}$ covers $V$, the choice
$\lambda_1=\cdots=\lambda_m=1$ is a fractional covering.
Let $\tau$ denote the minimum cost of a fractional covering of $V$.

Next, we consider the dual problem, in which weights are assigned to the
points of $V$ rather than to the sets in $\mathcal{E}$.
A collection of nonnegative numbers $q_1,\dots,q_r$ is called a
\textbf{fractional packing of $V$} with respect to $\mathcal{E}$ and
$(a_1, \dots, a_m)$ if
\[
\sum_{\alpha:\,v_\alpha\in E_j}q_\alpha\leq a_j
\qquad (1\leq j\leq m),
\]
where the sum is taken over all $\alpha$ with $v_\alpha \in E_j$.
Thus, the total weight assigned to the points of each $E_j$ is at most
its cost $a_j$.
Let $\nu$ denote the maximum of
$\sum_{\alpha=1}^r q_\alpha$ over all fractional packings of $V$.

\begin{proposition} \label{proposition: duality between tau and nu}
We have $\tau = \nu$.
\end{proposition}

\begin{proof}
Let $A=(A_{\alpha j})_{1\leq\alpha\leq r,\,1\leq j\leq m}$ be the
incidence matrix of $V$ and $\mathcal{E}$, defined by
\[
A_{\alpha j}
=
\begin{cases}
1 & \text{if }v_\alpha\in E_j,\\
0 & \text{otherwise}.
\end{cases}
\]
Write
\[
\lambda=(\lambda_1,\dots,\lambda_m)^T,
\qquad
q=(q_1,\dots,q_r)^T,
\qquad
a=(a_1,\dots,a_m)^T,
\]
and let $\mathbf{1}=(1,\dots,1)^T\in\mathbb{R}^r$.
Then the fractional covering problem is
\[
\tau
=
\min\{a^T\lambda\mid\lambda\geq0,\ A\lambda\geq\mathbf{1}\},
\]
while the fractional packing problem is
\[
\nu
=
\max\{\mathbf{1}^Tq\mid q\geq0,\ A^Tq\leq a\}.
\]
These are a primal--dual pair of linear programs.
The covering problem is feasible because
$\lambda=(1,\dots,1)^T$ is admissible, and its objective function is
bounded below by $0$.
Theorem \ref{theorem: linear programming duality}, applied with
$b=\mathbf{1}$ and $c=a$, therefore shows that $\tau=\nu$.
\end{proof}

\subsection{Finitary dynamical Frostman\rq{}s lemma}
\label{subsection: finitary dynamical Frostman lemma}

In this subsection, we prove a finitary version of the dynamical
Frostman lemma \cite[Lemma 3.3]{Feng--Huang} for arbitrary subsets
$\Omega\subset X$.
We first recall the fractional analogue of
$\Lambda^{w,s}_{N,\varepsilon}(\Omega)$, introduced by Feng and Huang
\cite[\S 3.2]{Feng--Huang}.

Let $0\leq w \leq 1$.
Let $\pi\colon (X, T)\to (Y, S)$ be 
an equivariant continuous map between dynamical systems.
We take metrics $\mathbf{d}$ and $\mathbf{d}^\prime$ on $X$ and $Y$ respectively.
Let $\Omega \subset X$ be an arbitrary subset.

For $s\geq0$, $\varepsilon>0$, and a positive integer $N$, define
$\mathcal W^{w,s}_{N,\varepsilon}(\Omega)$ to be the infimum of
\[
\sum_j c_j e^{-sn_j}
\]
over all at most countable families $(x_j,n_j,c_j)$, where $x_j\in X$,
$n_j\geq N$ is an integer, and $c_j\geq0$, satisfying
\[
\mathbf 1_{\Omega}
\leq
\sum_j c_j\mathbf 1_{B^w_{n_j}(x_j,\varepsilon)}
\qquad\text{pointwise on }X.
\]
Here $\mathbf{1}_{\Omega}$ and 
$\mathbf{1}_{B^w_{n_j}(x_j, \varepsilon)}$ denote the 
characteristic functions of $\Omega$ and 
$B^w_{n_j}(x_j, \varepsilon)$ respectively.

The next proposition is a key result.
This was given in \cite[Proposition 3.5]{Feng--Huang}.
We notice that the proof of \cite[Proposition 3.5]{Feng--Huang} does not utilize 
measure theory although it is strongly motivated by the ideas of 
geometric measure theory.

\begin{proposition} 
  \label{proposition: fractional version of weighted entropy definition}
For any $s\geq 0$, $\varepsilon>0$ and $\delta>0$, if we choose 
$N$ sufficiently large then 
\[ \Lambda^{w, s+\delta}_{N, 6\varepsilon}(\Omega) \leq
    \mathcal{W}^{w,s}_{N,\varepsilon}(\Omega)
    \leq \Lambda^{w, s}_{N,\varepsilon}(\Omega). \]
\end{proposition}

The following lemma is the main result of this subsection.

\begin{lemma}[Finitary dynamical Frostman\rq{}s lemma] 
  \label{lemma: finitary dynamical Frostman lemma}
  Let $0<s < \FHh^w(\Omega, T)$.
  There exist $r>0$ and a positive integer $N_0$ such that, for every
  integer $L\geq N_0$, there are a finite subset $E_L\subset\Omega$
  and a function $p_L\colon E_L\to[0,1]$ satisfying the following
  conditions:
   \begin{itemize}
    \item $\sum_{x\in E_L} p_L(x) = 1$,
    \item for any integer $n$ with $N_0\leq n \leq L$ and a subset $A\subset X$, 
          if the diameter of 
          $A$ with respect to $\mathbf{d}^w_n$ is smaller than $r$ then we have
          \[ \sum_{x\in A\cap E_L} p_L(x)  \leq e^{-sn}. \]
   \end{itemize}
\end{lemma}
We write 
$\displaystyle p_L(A):=\sum_{x\in A\cap E_L}p_L(x)$
for any subset $A$ of $X$ below.
Here, \lq\lq{}finitary\rq\rq{} means that $p_L$ 
provides a probability measure supported on a finite subset of $X$
and that the Frostman-type
estimates are required only for the finite range of times
$N_0\leq n\leq L$.
In the original proof of \cite[Lemma 3.3]{Feng--Huang}, 
Feng and Huang used the Hahn--Banach theorem and 
the Riesz representation theorem.
In essence, our proof below replaces these 
infinite-dimensional functional-analytic tools 
with finite-dimensional linear programming.

\begin{proof}
Take $\delta>0$ with $s+\delta < \FHh^w(\Omega, T)$.
Choose $r>0$ so that $s+\delta < \FHh^w(\Omega, T, 12r)$.
Then we have $\Lambda^{w, s+\delta}_{N, 12r}(\Omega)>1$ 
for sufficiently large $N$.
By Proposition \ref{proposition: fractional version of weighted entropy definition},
we can take a positive integer $N_0$ that satisfies 
\[  \mathcal{W}^{w,s}_{N_0,2r}(\Omega) >1. \]

Fix $L\geq N_0$.
For any integer $n$ with $N_0\leq n \leq L$, we take a finite subset 
$F_n\subset X$ such that $F_n$ is an $r$-spanning subset of $X$ 
with respect to $\mathbf{d}^w_n$,
namely for any $x\in X$ there exists $z\in F_n$ with 
$\mathbf{d}^w_n(x,z)< r$.
Set 
\[ \mathcal{B}_L = \{B_n^w(z, 2r)\mid z\in F_n, \, N_0\leq n \leq L\}. \] 
For $B = B_n^w(z, 2r)$, we write 
$\mathrm{n}(B) = n$.\footnote{Strictly speaking, 
$n$ is not necessarily determined by
the set $B=B_n^w(z,2r)$ itself.
We therefore regard $\mathcal B_L$ as an indexed family of balls
labeled by the pairs $(n,z)$ with $N_0\leq n\leq L$ and $z\in F_n$.
Distinct labels are retained even when they define the same
subset of $X$.
For simplicity, we suppress the labels in the notation;
all sums over $\mathcal B_L$ are understood to run over these labels.}

We can take 
a finite subset\footnote{Since $\mathcal B_L$ is finite, 
only finitely many membership patterns
\[
\bigl(\mathbf 1_B(x)\bigr)_{B\in\mathcal B_L},
\qquad x\in\Omega,
\]
occur.
Choose one representative of each pattern occurring in $\Omega$,
and let $E_L\subset\Omega$ be the set of these representatives.} 
$E_L\subset \Omega$ such that 
for any $x\in \Omega$ there exists $y\in E_L$ for which we have 
$\mathbf{1}_B(x) = \mathbf{1}_B(y)$ for all $B\in \mathcal{B_L}$.

We define 
\[ \tau_L 
  = \min\left\{\sum_{B\in \mathcal{B}_L} c_B \cdot e^{-\mathrm{n}(B)s}\, 
    \middle|\, 
   c_B\geq 0, \, \sum_{B:\, x\in B} c_B \geq 1 
   \text{ for all $x\in \Omega$}\right\}.\]
It follows from the definition of 
$\mathcal{W}^{w,s}_{N_0,2r}(\Omega)$ that 
\[ \tau_L \geq \mathcal{W}^{w,s}_{N_0,2r}(\Omega) >1. \]
From the definition of $E_L$ we also have 
\[ \tau_L = 
   \min\left\{\sum_{B\in \mathcal{B}_L} c_B \cdot e^{-\mathrm{n}(B)s}\, 
    \middle|\, 
   c_B\geq 0, \, \sum_{B:\, x\in B} c_B \geq 1 
   \text{ for all $x\in E_L$}\right\}.\]
Thus, $\tau_L$ is the value of the finite fractional covering problem
(reviewed in 
\S \ref{subsection: fractional covering problem and its dual}) 
on $E_L$ with the indexed family $(E_L\cap B)_{B\in\mathcal B_L}$
and costs $e^{-s\mathrm n(B)}$.

By Proposition \ref{proposition: duality between tau and nu}, 
there exists a map $q_L\colon E_L\to [0, \infty)$ such that 
$\sum_{z\in E_L} q_L(z) = \tau_L$ and 
\[ \sum_{z \in E_L\cap B} q_L(z) \leq e^{-\mathrm{n}(B)s} \quad 
   \text{ for all $B\in \mathcal{B}_L$}. \]

Define $p_L\colon E_L\to [0,1]$ by $p_L(z) = q_L(z)/\tau_L$.
We have $p_L(\Omega) = 1$ and $p_L(B) \leq e^{-\mathrm{n}(B)s}$
for all $B\in \mathcal{B}_L$.

Suppose $A\subset X$ satisfies $\diam (A, \mathbf{d}^w_n) < r$
for some $n$ with $N_0\leq n \leq L$.
Since $F_n$ is $r$-spanning with respect to $\mathbf{d}^w_n$, 
there is $z\in F_n$ with $A\subset B^w_n(z, 2r)$.
Then 
\[ p_L(A) \leq p_L\left(B^w_n(z, 2r)\right) \leq e^{-sn}. \]
\end{proof}

\subsection{Partition construction and entropy estimates}
\label{subsection: partition construction and entropy estimates}

This subsection is rather technical, and the motivation for the
constructions below may not be immediately apparent.
We therefore begin by explaining the underlying ideas.

Let $(X,T)$ be a dynamical system, and recall the standard
variational principle
\[
\h(X,T)
=
\sup_{\mu\in\mathscr{M}^T(X)}h_\mu(T).
\]
Topological entropy is defined using finite open covers, whereas
measure-theoretic entropy is defined using finite measurable partitions.
If $X$ is zero-dimensional, one can choose finite clopen partitions
with arbitrarily small atom diameters.
Such partitions serve simultaneously as open covers and measurable
partitions.
For a general compact metric space, however, such partitions need
not exist.

This difference between open covers and measurable partitions
is a source of technical difficulty in the proof of the variational
principle.
Near the boundaries of partition atoms, arbitrarily close points
may belong to different atoms.
This can introduce error terms when comparing covering numbers
with partition entropies.

A similar difficulty arises in the proof of
Theorem \ref{theorem: main theorem}.\footnote{If both $X$ and $Y$
are zero-dimensional (in particular, if they are subshifts over
finite alphabets), this boundary difficulty can be avoided by
using clopen partitions, and the proof of
Theorem \ref{theorem: main theorem} becomes substantially simpler.}
The partitions constructed in Lemma 
\ref{lemma: Partition families with controlled exceptional sets} 
below are designed to control these
boundary effects.
The key idea is to use a finite family of partitions rather than
a single partition.
The lemma 
provides partitions with atoms of small diameter,
a closed exceptional set for each partition, and a common
radius $\delta>0$.
For each partition, every ball of radius $\delta$ centered outside
its exceptional set is contained in a single atom of that partition.

We also require each point to be exceptional for only a small
proportion of the partitions.
More precisely, for any prescribed $\eta>0$, the family and the
radius $\delta$ can be chosen so that every point belongs to at most
an $\eta$-fraction of the exceptional sets.
A finite averaging argument then shows that, for any finitely
supported probability measure, there is a partition whose
exceptional set has measure at most $\eta$.
The family and the radius $\delta$ are chosen independently of
the measure; only the selection of a partition from the family
depends on it.
Thus, the same family and radius can be used for all the finitely
supported probability measures arising in our proof of Theorem 
\ref{theorem: main theorem}.

Finally, once the metric space and the prescribed bound on atom
diameters are fixed, the number of atoms in each partition is
bounded independently of $\eta$.
This uniform bound will be important when controlling the error
terms in the entropy estimates below.

Although motivated by applications to dynamical systems, the following
result is formulated for an arbitrary compact metric space and
does not require any underlying dynamics.

\begin{lemma}[Partition families with controlled exceptional sets]
\label{lemma: Partition families with controlled exceptional sets}
Let $(Z,\mathbf{d})$ be a compact metric space. 
For every $a>0$, there exists
an integer $q\geq 1$ with the following property.

For every $\eta>0$, there exist an integer $M\geq 1$, a number
$\delta>0$, finite partitions
\[
    \mathcal P_1,\ldots,\mathcal P_M
\]
of $Z$, and closed sets
\[
    D_1,\ldots,D_M\subset Z
\]
satisfying the following conditions:
\begin{enumerate}
    \item[(1)]
    Each partition $\mathcal P_m$ has at most $q$ atoms,
    all of diameter less than $a$.

    \item[(2)]
    For every $1\leq m \leq M$ and every $z\in Z\setminus D_m$,
    the open ball $B(z,\delta)$ is contained in a single atom
    of $\mathcal P_m$.

    \item[(3)]
    Every point belongs to at most an $\eta$-fraction of the
    exceptional sets:
    \[
        \#\{m\in\{1,\ldots,M\}:z\in D_m\}
        \leq \eta M
        \qquad (z\in Z).
    \]
\end{enumerate}
In particular, the bound $q$ on the number of atoms is independent
of $\eta$.
\end{lemma}

\begin{proof}
Fix $a>0$.
By compactness, we can choose points $z_1,\dots,z_q\in Z$ such that
\[
Z=\bigcup_{i=1}^q B(z_i,a/4).
\]
The integer $q$ is fixed throughout the proof and does not depend
on $\eta$.

Let $\eta>0$.
Choose a positive integer $M$ such that
\[
\frac{q}{M}\leq\eta.
\]
Choose distinct radii $r_1,\dots,r_M\in(a/4,a/2)$, and then choose
$\delta>0$ sufficiently small that the closed intervals
\[
[r_m-\delta,r_m+\delta],
\qquad 1\leq m\leq M,
\]
are pairwise disjoint.

For $1\leq m\leq M$ and $1\leq i\leq q$, set
\[
U_{m,i}:=B(z_i,r_m),
\qquad
A_{m,i}:=
U_{m,i}\setminus\bigcup_{j=1}^{i-1}U_{m,j},
\]
where $A_{m, 1} = U_{m,1}$.
Since $r_m>a/4$, the sets $U_{m,1},\dots,U_{m,q}$ cover $Z$.
Thus, $\mathcal P_m := \{A_{m,1},\dots,A_{m,q}\}$ 
is a finite partition of $Z$.
Each atom $A_{m,i}$ is contained in the ball $U_{m,i}$, so
\[
\operatorname{diam}(A_{m,i},\mathbf d)\leq 2r_m<a.
\]
This proves (1).

For each $1\leq m\leq M$, define
\[
D_m:=
\bigcup_{i=1}^q
\left\{
z\in Z  \, \middle|\, 
\left|\mathbf d(z,z_i)-r_m\right|\leq\delta
\right\}.
\]
We next verify (2).
Fix $m$ and $z\in Z\setminus D_m$, and take any
$z'\in B(z,\delta)$.
For every $1\leq i\leq q$, we have
\[
\left|\mathbf d(z,z_i)-r_m\right|>\delta
\]
and, by the triangle inequality,
\[
\left|\mathbf d(z',z_i)-\mathbf d(z,z_i)\right|
\leq\mathbf d(z,z')<\delta.
\]
Consequently, $\mathbf d(z',z_i)$ and $\mathbf d(z,z_i)$ lie
on the same side of $r_m$, and hence
\[
z'\in U_{m,i}
\quad\Longleftrightarrow\quad
z\in U_{m,i}.
\]
Thus, $z'$ and $z$ belong to the same atom of $\mathcal P_m$.
Since $z'\in B(z,\delta)$ was arbitrary, the whole ball
$B(z,\delta)$ is contained in the atom containing $z$.
This proves (2).

Finally, fix $z\in Z$.
For each $1\leq i\leq q$, the pairwise disjointness of the
intervals $[r_m-\delta,r_m+\delta]$ implies that
\[
\left|\mathbf d(z,z_i)-r_m\right|\leq\delta
\]
holds for at most one index $m$.
Since there are $q$ centers $z_i$, 
the point $z$ belongs to at most $q$ of the
sets $D_1,\dots,D_M$.
Therefore,
\[
\#\{m\in\{1,\dots,M\}:z\in D_m\}
\leq q\leq\eta M.
\]
This proves (3) and completes the proof.
\end{proof}

We next derive entropy estimates from the local containment property
in Lemma \ref{lemma: Partition families with controlled exceptional sets} (2).
These estimates quantify the error caused by the exceptional set.

Let $(Z, \mathbf{d})$ be a compact metric space, 
$E\subset Z$ a finite subset, and 
$p\colon E \to [0,1]$ a map satisfying $\sum_{z\in E}p(z) = 1$.
We regard $p$ as a finitely supported probability measure on $Z$;
for every subset $A\subset Z$, we write
\[
p(A):=\sum_{z\in E\cap A}p(z).
\]
Let $\mathcal{P}$ be a finite partition of $Z$.
For $z\in Z$, we denote by $\mathcal{P}(z)$ 
the atom of $\mathcal{P}$ containing $z$.
Let $I\colon E\to \{1,2,\dots, K\}$ be a map.
We define
the \textbf{conditional Shannon entropy of $\mathcal P$ given $I$} by
\[
H_p(\mathcal P\mid I)
:=
\sum_{\substack{1\leq k\leq K \text{ with}\\p(I^{-1}(k))>0}}
p\left(I^{-1}(k)\right)\,
H_p(\mathcal P\mid I=k),
\]
where the term 
$H_p(\mathcal{P}\mid I=k)$ is defined in the following (standard) way:
Let $\mathcal{P} = \{A_1, \dots, A_a\}$. Then 
\[ H_p(\mathcal{P}\mid I=k) := 
    -\sum_{i=1}^a 
    \frac{p\left(A_i\cap I^{-1}(k)\right)}
    {p\left(I^{-1}(k)\right)} 
    \log  \frac{p\left(A_i\cap I^{-1}(k)\right)}
    {p\left(I^{-1}(k)\right)}.
    \] 
The readers can find basics of Shannon entropy in the book 
of Cover--Thomas \cite[Chapter 2]{Cover--Thomas}.

\begin{lemma} \label{lemma: entropy estimates}
 Let $\delta>0$, and let $D\subset Z$ be a subset such that
 the open ball $B(z,\delta)$ is contained in the atom $\mathcal{P}(z)$
 for every $z\in Z\setminus D$.
  \begin{enumerate}
   \item[(1)] If $\diam\left(I^{-1}(k), \mathbf{d}\right) < \delta$
              for all $k\in \{1,2,\dots, K\}$ then 
              \[ H_p(\mathcal{P}\mid I) \leq p(D) \log |\mathcal{P}|, \]
              where $|\mathcal{P}|$ denotes the number of atoms of 
              $\mathcal{P}$.
   \item[(2)] Suppose we are given a continuous map $R\colon Z\to Z$ 
              (so that $(Z, R)$ is a dynamical system).
              Let $n$ be a positive integer, and
              let $\mathcal{P}^n := \bigvee_{i=0}^{n-1} R^{-i}\mathcal{P}$.
              Suppose 
              that the diameter of $I^{-1}(k)$ with respect to the Bowen 
              metric $\mathbf{d}_n$ is smaller than $\delta$ for every 
              $k \in \{1,2,\dots, K\}$. Then 
              \[ H_p\left(\mathcal{P}^n \, \middle|\, I\right) 
                 \leq \log |\mathcal{P}| \sum_{i=0}^{n-1} p(R^{-i}D). \]
  \end{enumerate}
\end{lemma}  

\begin{proof}
(1) If $I^{-1}(k)$ is not contained in $D$, then 
$\mathcal{P}(z) = \mathcal{P}(z^\prime)$ for all 
$z, z^\prime \in I^{-1}(k)$ and hence 
$H_p(\mathcal{P}|I=k) = 0$.
Consequently, only fibers entirely contained in $D$ can contribute
to the conditional entropy.
Therefore 
\[ H_p(\mathcal{P}|I) = \sum_{k:\, I^{-1}(k)\subset D}
   p\left(I^{-1}(k)\right) 
   \underbrace{H_p\left(\mathcal{P}\middle|\, I=k\right)}_
   {\leq \log |\mathcal{P}|} \leq
   p(D) \log |\mathcal{P}|. \]
(2) It follows from the subadditivity of conditional entropy that 
\[ H_p\left(\mathcal{P}^n\middle|\, I\right) \leq
   \sum_{i=0}^{n-1} H_p\left(R^{-i}\mathcal{P}\middle|\, I\right). \]
Let $0\leq i < n$.
If $I^{-1}(k)$ is not contained in $R^{-i}D$, then 
$\mathcal{P}(R^i z) = \mathcal{P}(R^i z^\prime)$
for all $z, z^\prime\in I^{-1}(k)$ and hence 
$H_p\left(R^{-i}\mathcal{P}\middle|\, I=k\right) =0$.
Therefore, as in (1)
\[ H_p\left(R^{-i}\mathcal{P}\middle|\, I\right) 
   \leq  p\left(R^{-i}D\right) \log |\mathcal{P}|. \]
   Summing over $i=0,\dots,n-1$ proves (2).
\end{proof}

\subsection{Proof of Theorem \ref{theorem: main theorem}}
\label{subsection: proof of main theorem}

Now we are ready to prove Theorem \ref{theorem: main theorem}
for general $\Omega$.
Let $\pi\colon (X, T)\to (Y, S)$ be an equivariant continuous map 
between dynamical systems.
We take metrics $\mathbf{d}$ and $\mathbf{d}^\prime$ on $X$ and $Y$,
respectively.

Let $0\leq w \leq 1$, 
and let $\Omega \subset X$ be an arbitrary (nonempty) subset.
We would like to prove 
\begin{equation}
 \label{eq: main theorem}
  \FHh^w(\Omega, T) \leq \h^w(\Omega,T). 
\end{equation}

First we consider the case of $w=0$.
As in \S 
\ref{section: proof of the main theorem for Omega = X},
$\h^0(\Omega, T) = \h\left(\pi(\Omega),S\right)$
is the topological entropy of $\pi(\Omega) \subset Y$.
A cover of $\pi(\Omega)$ by open sets of $\mathbf{d}^\prime_n$-diameter 
less than $\varepsilon$ gives a cover of $\Omega$ by the 
same number of $0$-weighted Bowen balls of length $n$ 
and radius $\varepsilon$. Therefore 
\[ \FHh^0\left(\Omega, T\right) \leq
    \h^0\left(\Omega, T\right). \]
Hence we assume $0< w\leq 1$ in the rest of this section. 

If $\FHh^w(\Omega, T)=0$ then \eqref{eq: main theorem} is trivial.
So we assume $\FHh^w(\Omega, T) >0$.
Let $s$ be an arbitrary positive number with 
$s< \FHh^w(\Omega,T)$.
Let $\tau>0$ be an arbitrary positive number.
We would like to prove 
\[ \h^w(\Omega, T) \geq  s-2\tau. \]
Once this is proved, we obtain the conclusion \eqref{eq: main theorem}
by letting $\tau\to 0$ and $s\to \FHh^w(\Omega,T)$.

We divide the remaining argument into six steps.

\medskip
\noindent\textbf{Step 1. Fixing the parameters and partition families.}
By finitary dynamical Frostman\rq{}s lemma 
(Lemma \ref{lemma: finitary dynamical Frostman lemma}), 
choose
$r>0$ and a positive integer $N_0$ such that, 
for every $L\geq N_0$,
there is a finitely supported probability measure 
$p_L$ on $\Omega$
with
\begin{equation}
\label{eq: general proof Frostman bound}
p_L(A)\leq e^{-sn}
\quad\text{if }N_0\leq n\leq L
\text{ and }\diam(A,\mathbf d_n^w)<r.
\end{equation}
Apply the partition lemma (Lemma
\ref{lemma: Partition families with controlled exceptional sets})
to $(X,\mathbf d)$ and $(Y,\mathbf d')$, with the diameter bound
$a=r/2$, and let $q_X$ and $q_Y$ be the resulting bounds on the
numbers of atoms. 
These bounds are independent of the parameter $\eta$ chosen below.
Set
\[
C:=1+\log q_X+\log q_Y.
\]
Choose a positive integer $K$ and then $\eta>0$ such that
\begin{equation}
\label{eq: general proof fixed errors}
\frac{wC}{K}<\frac{\tau}{2},
\qquad
\frac{C\eta}{w^K}<\tau.
\end{equation}
With this choice of $\eta$, 
the partition lemma 
(Lemma \ref{lemma: Partition families with controlled exceptional sets})
gives finite partitions
\[
\mathcal P_1,\dots,\mathcal P_{M_X}
\quad\text{on }X,
\qquad
\mathcal Q_1,\dots,\mathcal Q_{M_Y}
\quad\text{on }Y,
\]
closed exceptional sets $D^X_m$ and $D^Y_m$, and radii
$\delta_X,\delta_Y>0$ satisfying its conclusions.
Set
\[
\delta:=\min\{\delta_X,\delta_Y\}.
\]
Each $\mathcal P_m$ has at most $q_X$ atoms, each $\mathcal Q_m$
has at most $q_Y$ atoms, and all these atoms have diameter less
than $r/2$ with respect to $\mathbf{d}$ and $\mathbf{d}^\prime$ 
respectively.
All the parameters and families chosen so far are independent of 
the parameter $L$.
(The probability measure $p_L$ may depend on $L$, but
we have not chosen it so far.)

\medskip
\noindent\textbf{Step 2. Choosing partitions for a finite Frostman measure.}
Let $L$ be any sufficiently large positive integer such that
\begin{equation}
\label{eq: general proof choice of L}
w^K L\geq N_0,
\qquad
\frac{C}{w^K L}<\frac{\tau}{2}.
\end{equation}
Take $E_L\subset\Omega$ and $p_L$ as in
finitary dynamical Frostman\rq{}s lemma
(Lemma \ref{lemma: finitary dynamical Frostman lemma}), 
and write
$E=E_L$ and $p=p_L$ for brevity.

\begin{claim}
\label{claim: general proof partition selection}
There are partitions $\mathcal{P}_m$ and $\mathcal{Q}_{m^\prime}$
$(1\leq m \leq M_X$, $1\leq m^\prime \leq M_Y$) such that 
the corresponding exceptional sets $D^X_{m}$ and $D^Y_{m^\prime}$ satisfy
\begin{equation}
\label{eq: general proof boundary visits}
\sum_{i=0}^{L-1}p(T^{-i}D^X_m)\leq\eta L,
\qquad
\sum_{i=0}^{L-1}p(\pi^{-1}S^{-i}D^Y_{m^\prime})\leq\eta L.
\end{equation}
\end{claim}

\begin{proof}
For the family on $X$, condition (3) of the partition lemma gives
\begin{align*}
\frac1{M_X}\sum_{m=1}^{M_X}\sum_{i=0}^{L-1}p(T^{-i}D^X_m)
&=\sum_{i=0}^{L-1}\sum_{x\in E}p(x)
  \frac1{M_X}\sum_{m=1}^{M_X}\mathbf1_{D^X_m}(T^ix)\\
&\leq\eta L.
\end{align*}
Hence at least one member of this family satisfies the first
inequality in \eqref{eq: general proof boundary visits}.
The same calculation for the family on $Y$, using $S^i\pi(x)$
in place of $T^ix$, gives the second inequality.
\end{proof}

Fix these partitions $\mathcal{P}_m$ and $\mathcal{Q}_{m^\prime}$
for the rest of the argument with this $L$.
Write $\mathcal{P} = \mathcal{P}_m$, 
$\mathcal{Q} = \mathcal{Q}_{m^\prime}$,
$D_X = D^X_m$ and $D_Y = D^Y_{m^\prime}$ for brevity.
We note that 
the partitions $\mathcal P,\mathcal Q$ and the exceptional sets
$D_X,D_Y$ may depend on $L$ and $p$, 
but the bounds on the numbers of atoms and
the radius $\delta$ do not.

So far, we have introduced several parameters and partitions.
The order in which they are chosen is important but may not be
immediately clear.
We therefore summarize the choices made in Steps~1 and~2
in the table below.
In making these choices, we keep the dynamical systems $(X,T)$
and $(Y,S)$, the map $\pi$, the subset $\Omega$, 
and the parameters $w,s,\tau$ fixed.

\begin{center}
\small
\renewcommand{\arraystretch}{1.3}
\setlength{\tabcolsep}{5pt}
\begin{tabularx}{\linewidth}{@{}c
  >{\raggedright\arraybackslash}p{0.25\linewidth}
  >{\raggedright\arraybackslash}X@{}}
\toprule
\textbf{Order} & \textbf{Objects chosen} & \textbf{Choice or requirement} \\
\midrule
\multicolumn{3}{@{}l}{\emph{Choices fixed as $L\to\infty$}} \\[0.3ex]

1 & $r,N_0$
  & Finitary Frostman lemma with exponent $s$. \\[0.4ex]

2 & $q_X,q_Y,C$
  & Atom bounds for diameter less than $r/2$, independent of $\eta$;
    $C=1+\log q_X+\log q_Y$. \\[0.4ex]

3 & $K$
  & Chosen so that $wC/K<\tau/2$. \\[0.4ex]

4 & $\eta$
  & Chosen so that $C\eta/w^K<\tau$. \\[0.4ex]

5 & Partition families, exceptional sets, and $\delta$
  & The partition lemma with diameter bound $r/2$ and overlap
    bound $\eta$; $\delta=\min\{\delta_X,\delta_Y\}$. \\

\midrule
\multicolumn{3}{@{}l}{\emph{The large time parameter and the choices depending on it}} \\[0.3ex]

6 & $L$
  & Any sufficiently large integer satisfying $w^K L\geq N_0$
    and $C/(w^K L)<\tau/2$. \\[0.4ex]

7 & $E_L,p_L$
  & Frostman probability measure $p_L$ supported on a finite set
    $E_L\subset\Omega$, for the time range $N_0\leq n\leq L$. \\[0.4ex]

8 & $\mathcal P,\mathcal Q,D_X,D_Y$
  & Selected using $p_L$ so that each of the two boundary-visit
    sums is at most $\eta L$. \\

\bottomrule
\end{tabularx}
\end{center}

In particular, the common radius $\delta$ is fixed before $L$ and
$p_L$ are chosen, although the selected partitions and their
exceptional sets may depend on both.
This distinction will be essential when we take the upper limit
in time at the end of the proof.

\medskip
\noindent\textbf{Step 3. Entropy at two different time scales.}
For every positive integer $n$, write
\[
\mathcal P^n:=\bigvee_{i=0}^{n-1}T^{-i}\mathcal P,
\qquad
\mathcal Q^n:=\bigvee_{i=0}^{n-1}S^{-i}\mathcal Q,
\qquad
\mathcal A_n:=\pi^{-1}(\mathcal Q^n).
\]
For $1\leq n\leq L$, set
\begin{equation}
\label{eq: general proof entropy sequences}
a_n:=H_p(\mathcal A_n),
\qquad
b_n:=H_p(\mathcal P^n\mid\mathcal A_n).
\end{equation}
These quantities are nonnegative, and
\begin{equation}
\label{eq: general proof linear bound}
0\leq b_n\leq H_p(\mathcal P^n)
\leq n\log|\mathcal P|\leq Cn.
\end{equation}

\begin{claim}
\label{claim: general proof mixed entropy}
For every integer $n$ with $N_0\leq n\leq L$, we have
\[
a_n+b_{\lceil wn\rceil}\geq sn.
\]
\end{claim}

\begin{proof}
Set $m=\lceil wn\rceil$ and consider the partition
\[
\mathcal R_n:=\mathcal P^m\vee\mathcal A_n.
\]
Every atom of $\mathcal R_n$ has $\mathbf d_n^w$-diameter less
than $r$. 
Thus, \eqref{eq: general proof Frostman bound} gives
$p(A)\leq e^{-sn}$ for every atom $A$ of $\mathcal R_n$.
It follows that
\[
H_p(\mathcal R_n)
=\sum_{\substack{A\in\mathcal R_n\\p(A)>0}}
  p(A)\log\frac1{p(A)}
\geq sn.
\]
Since $m\leq n$, the partition $\mathcal A_n$ refines
$\mathcal A_m$. Conditioning on this finer partition cannot
increase entropy, so
\begin{align*}
H_p(\mathcal R_n)
&=H_p(\mathcal A_n)+H_p(\mathcal P^m\mid\mathcal A_n)\\
&\leq H_p(\mathcal A_n)+H_p(\mathcal P^m\mid\mathcal A_m)\\
&=a_n+b_m.
\end{align*}
This proves the claim.
\end{proof}

\medskip
\noindent\textbf{Step 4. Selecting a time scale.}

\begin{claim}
\label{claim: general proof good scale}
There is an integer $n$ with $w^K L\leq n\leq L$ such that
\[
\frac{a_n+wb_n}{n}>s-\tau.
\]
\end{claim}

\begin{proof}
Notice that the estimate in
Claim \ref{claim: general proof mixed entropy} is available only
for $n\geq N_0$. 
To apply the scale-descent inequality 
(Lemma \ref{lemma: scale-descent inequality}) as stated, 
set\footnote{This auxiliary modification 
is only needed to apply the lemma
in its stated form. It can be omitted by using the observation
at the end of \S\ref{subsection: scale-descent inequality}.}
\[
\widetilde a_n:=
\begin{cases}
sn & (1\leq n<N_0),\\
a_n & (N_0\leq n\leq L).
\end{cases}
\]
Then $(\widetilde a_n)_{1\leq n\leq L}$ and $(b_n)_{1\leq n\leq L}$ are nonnegative,
$b_n\leq Cn$, and
\[
\widetilde a_n+b_{\lceil wn\rceil}\geq sn
\qquad (1\leq n\leq L).
\]
Define $n_0=L$ and $n_{k+1}=\lceil wn_k\rceil$ for $0\leq k<K$.
Lemma \ref{lemma: scale-descent inequality} gives an index $k<K$
such that
\[
\frac{\widetilde a_{n_k}+wb_{n_k}}{n_k}
\geq s-\frac{wC}{K}-\frac{C}{w^K L}>s-\tau.
\]
Since $n_k\geq w^K L\geq N_0$, we have
$\widetilde a_{n_k}=a_{n_k}$. Taking $n=n_k$ proves the claim.
\end{proof}

\medskip
\noindent\textbf{Step 5. Comparing entropy with two-step covering numbers.}

\begin{claim}
\label{claim: general proof covering comparison}
For every integer $n$ with $1\leq n\leq L$, we have
\begin{align}
\label{eq: general proof covering comparison}
a_n+wb_n
&\leq \log\#^w(\Omega,n,\delta)
 +w\log|\mathcal P|\sum_{i=0}^{n-1}p(T^{-i}D_X)\notag\\
&\hspace{18mm}
 +\log|\mathcal Q|\sum_{i=0}^{n-1}p(\pi^{-1}S^{-i}D_Y)\\
&\leq \log\#^w(\Omega,n,\delta)+C\eta L.\notag
\end{align}
\end{claim}

\begin{proof}
Consider any finite open cover $V_1,\dots,V_\ell$ of $\pi(\Omega)$
with $\diam(V_j,\mathbf d'_n)<\delta$.
Discard sets disjoint from $\pi(\Omega)$ and set
\[
t_j:=\#(\Omega\cap\pi^{-1}(V_j),\mathbf d_n,\delta)\geq1.
\]
Choose open subsets $U_{j1},\dots,U_{jt_j}$ of $X$ such that
\[
\Omega\cap\pi^{-1}(V_j)\subset\bigcup_{k=1}^{t_j}U_{jk},
\qquad
\diam(U_{jk},\mathbf d_n)<\delta.
\]
For each $y\in\pi(E)$, 
choose an index $I_Y(y)\in \{1,2,\dots,\ell\}$ with
$y\in V_{I_Y(y)}$, and define $I(x):=I_Y(\pi(x))$ for $x\in E$.
Next, choose $J(x)\in\{1,\dots,t_{I(x)}\}$ such that
$x\in U_{I(x),J(x)}$.
Thus, the fibers of $I_Y$ have $\mathbf d'_n$-diameter less
than $\delta$, and the fibers of $(I,J)$ have
$\mathbf d_n$-diameter less than $\delta$.

Let $\bar p=\pi_*p$ be the probability measure on the finite set
$\pi(E)$, so that $\bar p(B)=p(\pi^{-1}(B))$ for $B\subset Y$.
Applying Lemma \ref{lemma: entropy estimates} (2) on $Y$ gives
\begin{equation}
\label{eq: general proof base conditional error}
H_p(\mathcal A_n\mid I)
=H_{\bar p}(\mathcal Q^n\mid I_Y)
\leq \log|\mathcal Q|
\sum_{i=0}^{n-1}p(\pi^{-1}S^{-i}D_Y).
\end{equation}
The same lemma on $X$, applied to the finite-valued map $(I,J)$,
gives
\begin{equation}
\label{eq: general proof fiber conditional error}
H_p(\mathcal P^n\mid I,J)
\leq\log|\mathcal P|\sum_{i=0}^{n-1}p(T^{-i}D_X).
\end{equation}
By the chain rule, conditional subadditivity, and the fact that
conditioning cannot increase entropy,
\begin{align*}
H_p(\mathcal P^n\vee\mathcal A_n)
&\leq H_p(I,J)+H_p(\mathcal P^n\vee\mathcal A_n\mid I,J)\\
&\leq H_p(I,J)+H_p(\mathcal P^n\mid I,J)
             +H_p(\mathcal A_n\mid I),\\
H_p(\mathcal A_n)
&\leq H_p(I)+H_p(\mathcal A_n\mid I).
\end{align*}
Recall that
\[ a_n = H_p(\mathcal{A}_n), \qquad 
   b_n= H_p(\mathcal{P}^n\mid \mathcal{A}_n) = 
       H_p(\mathcal{P}^n \vee \mathcal{A}_n) - H_p(\mathcal{A}_n). \]
Since $0<w\leq1$, the above estimates imply
\begin{align} 
  \label{eq: general proof cover label entropy}
a_n+wb_n
&=wH_p(\mathcal P^n\vee\mathcal A_n)
  +(1-w)H_p(\mathcal A_n)\notag\\
&\leq w H_p(I,J) + (1-w) H_p(I) 
      +wH_p(\mathcal P^n\mid I,J)+H_p(\mathcal A_n\mid I)
      \notag \\
&= H_p(I)+wH_p(J\mid I)
  +wH_p(\mathcal P^n\mid I,J)+H_p(\mathcal A_n\mid I). 
\end{align}

It remains to bound the entropy of the cover labels.
Set $\lambda_j:=p(I=j)$.
Conditioned on $I=j$, the label $J$ takes at most $t_j$ values.
Hence, by concavity of the logarithm,
\begin{align*}
H_p(I)+wH_p(J\mid I)
&= \sum_{j:\, \lambda_j>0} \lambda_j \log \frac{1}{\lambda_j} + 
   w \sum_{j:\, \lambda_j>0} \lambda_j 
   \underbrace{H_p(J \mid I=j)}_{\leq \log t_j} \\
&\leq\sum_{j:\,\lambda_j>0}
   \lambda_j\log\frac{t_j^w}{\lambda_j}\\
&\leq\log\left(\sum_{j:\,\lambda_j>0}t_j^w\right)
\leq\log\left(\sum_{j=1}^\ell t_j^w\right).
\end{align*}
Combining this bound with
\eqref{eq: general proof base conditional error}--\eqref{eq: general proof cover label entropy}
and taking the infimum over the open covers $(V_j)$ proves the
first inequality in \eqref{eq: general proof covering comparison}.
For the second inequality, use $n\leq L$,
\eqref{eq: general proof boundary visits}, and
\[
w\log|\mathcal P|+\log|\mathcal Q|
\leq\log q_X+\log q_Y\leq C.
\]
\end{proof}

\medskip
\noindent\textbf{Step 6. Completion of the proof.}
Choose $n=n(L) \in [w^K L, L]$ as in Claim \ref{claim: general proof good scale}.
Claim \ref{claim: general proof covering comparison} gives
\[
\frac{\log\#^w(\Omega,n(L),\delta)}{n(L)}
\geq \frac{a_{n(L)}+wb_{n(L)}}{n(L)}
       -\frac{C\eta L}{n(L)}
>s-\tau-\frac{C\eta}{w^K}
>s-2\tau.
\]
This holds for every sufficiently large $L$.
The radius $\delta$ is independent of $L$, and
$n(L)\geq w^K L\to\infty$ as $L\to\infty$.
Therefore,
\[
\limsup_{n\to\infty}
\frac{\log\#^w(\Omega,n,\delta)}{n}\geq s-2\tau.
\]
By the definition of covering weighted topological entropy,
\[
\h^w(\Omega,T)\geq s-2\tau.
\]
Letting $\tau\to0$ and then letting $s$ increase to
$\FHh^w(\Omega,T)$ proves \eqref{eq: main theorem}.
This completes the proof of Theorem 
\ref{theorem: main theorem}.

\section{Example} \label{section: example}

The purpose of this section is to construct an equivariant continuous map 
$\pi\colon (X, T) \to (Y, S)$ between dynamical systems and 
a closed (noninvariant) subset $\Omega \subset X$
that satisfy 
\[ \hinf^w(\Omega, T) < \FHh^w(\Omega, T) < \h^w(\Omega, T) \]
for some weight $w\in (0,1)$.
In particular, this shows that $\h^w(\Omega,T)$ cannot be replaced by
$\hinf^w(\Omega,T)$ in Theorem \ref{theorem: main theorem}.

The basic ingredient is the following elementary combinatorial property of the set
\[
I
:=
\bigcup_{k=0}^{\infty}
\left([4^k,2\cdot 4^k)\cap\mathbb{N}\right)
=
\{1\}\cup\{4,5,6,7\}\cup\{16,17,\dots,31\}\cup\cdots.
\]

\begin{lemma} \label{lemma: combinatorial property of I}
Set
$A(n):=\left|I\cap\{1,2,\dots,n\}\right|$.
Then the following statements hold.
\begin{enumerate}
\item
\[
\liminf_{n\to\infty}\frac{A(n)}{n}
=
\frac{1}{3},
\qquad
\limsup_{n\to\infty}\frac{A(n)}{n}
=
\frac{2}{3}.
\]

\item
For every $n\geq1$,
\[
0
\leq
A(n)+2A\left(\left\lceil\frac{n}{2}\right\rceil\right)-n
\leq
2.
\]
Equivalently,
\[
A(n)+2A\left(\left\lceil\frac{n}{2}\right\rceil\right)
\in
\{n,n+1,n+2\}.
\]
\end{enumerate}
\end{lemma}

\begin{proof}
We first record an explicit formula for $A(n)$. For every $k\geq0$,
\begin{equation}
\label{eq:explicit-formula-A}
A(n)
=
\begin{cases}
\displaystyle
n-\frac{2}{3}(4^k-1),
&
4^k\leq n<2\cdot4^k,
\\[2mm]
\displaystyle
\frac{4^{k+1}-1}{3},
&
2\cdot4^k\leq n<4^{k+1}.
\end{cases}
\end{equation}
Indeed, in the first case,
\[
A(n)
=
\sum_{r=0}^{k-1}4^r+(n-4^k+1)\\
=
n-\frac{2}{3}(4^k-1).
\]
In the second case,
\[
A(n)
=
\sum_{r=0}^{k}4^r
=
\frac{4^{k+1}-1}{3}.
\]

We first prove (1). On each interval
$4^k\leq n<2\cdot4^k$,
the ratio $A(n)/n$ is increasing in $n$, whereas on each interval
$2\cdot4^k\leq n<4^{k+1}$,
it is decreasing in $n$. Therefore,
\[
\liminf_{n\to\infty}\frac{A(n)}{n}
=
\lim_{k\to\infty}
\frac{A(4^{k+1}-1)}{4^{k+1}-1}
=
\frac{1}{3},
\]
and
\[
\limsup_{n\to\infty}\frac{A(n)}{n}
=
\lim_{k\to\infty}
\frac{A(2\cdot4^k-1)}{2\cdot4^k-1} =
\lim_{k\to\infty}
\left(
1-\frac{2(4^k-1)}{3(2\cdot4^k-1)}
\right) = \frac{2}{3}.
\]

We next prove (2). Suppose first that
$4^k\leq n<2\cdot4^k$.
If $k=0$, then $n=1$, and the assertion is immediate. 
Assume that $k\geq1$. Then
\[
2\cdot4^{k-1}
\leq
\left\lceil\frac{n}{2}\right\rceil
\leq
4^k.
\]
The upper endpoint $4^k$ occurs precisely when
$n=2\cdot4^k-1$.
It follows from \eqref{eq:explicit-formula-A} that
\[
A\left(\left\lceil\frac{n}{2}\right\rceil\right)
=
\frac{4^k-1}{3}
+
\mathbf{1}_{\{n=2\cdot4^k-1\}},
\]
where $\mathbf{1}_{E}$ denotes the indicator of a condition $E$. Hence
\[
A(n)
+
2A\left(\left\lceil\frac{n}{2}\right\rceil\right)
-n
=
2\cdot \mathbf{1}_{\{n=2\cdot4^k-1\}}
\in
\{0,2\}.
\]

Next suppose that
$2\cdot4^k\leq n<4^{k+1}$.
Then
\[
4^k
\leq
\left\lceil\frac{n}{2}\right\rceil
\leq
2\cdot4^k.
\]
The upper endpoint $2\cdot4^k$ occurs precisely when
$n=4^{k+1}-1$.
Therefore, \eqref{eq:explicit-formula-A} gives
\[
A\left(\left\lceil\frac{n}{2}\right\rceil\right)
=
\left\lceil\frac{n}{2}\right\rceil
-
\frac{2}{3}(4^k-1)
-
\mathbf{1}_{\{n=4^{k+1}-1\}}.
\]
Consequently,
\[
A(n)
+
2A\left(\left\lceil\frac{n}{2}\right\rceil\right)
-n
=
2\left\lceil\frac{n}{2}\right\rceil
-n+1
-
2\cdot \mathbf{1}_{\{n=4^{k+1}-1\}}.
\]
If $n\neq4^{k+1}-1$, then
\[
2\left\lceil\frac{n}{2}\right\rceil-n
\in
\{0,1\},
\]
and hence the right-hand side belongs to $\{1,2\}$. If
$n=4^{k+1}-1$,
then $n$ is odd and
\[
2\left\lceil\frac{n}{2}\right\rceil-n=1.
\]
In this case, the right-hand side is equal to $0$. This proves (2).
\end{proof}

Let $X$ and $Y$ be the one-sided full-shifts 
on the alphabet $\{0,1,2,3\}\times \{0,1\}$ and $\{0,1\}$, respectively:
\[ X = \left(\{0,1,2,3\}\times \{0,1\}\right)^{\mathbb{N}}, \quad 
    Y = \{0,1\}^{\mathbb{N}}. \]
We denote the shift maps on $X$ and $Y$ by $T$ and $S$, respectively.
Let $\pi\colon X\to Y$ be the natural projection: 
\[ \pi\left((a_n, b_n)_{n\in \mathbb{N}}\right) = (b_n)_{n \in \mathbb{N}}, \qquad 
    (a_n\in \{0,1,2,3\}, b_n\in \{0,1\}). \]
We now use the set $I$ to specify the coordinates at which symbols may vary; 
we define a closed and noninvariant subset $\Omega \subset X$ by 
\[ \Omega = \{(x_n)_{n\in \mathbb{N}} \mid x_n 
= (0,0) \text{ for $n\not\in I$}\}. \]
For this example with $w = 1/2$, we will show
\[ \hinf^w(\Omega, T) < \FHh^w(\Omega, T) < \h^w(\Omega, T). \]
The mechanism behind this result is as follows:
The covering entropies reflect the oscillation of $A(n)/n$.
For the Feng--Huang entropy, however, the relevant cylinder count
involves the combination $A(n)+2A(\lceil n/2\rceil)$,
which differs from $n$ by a uniformly bounded amount.
In other words, although $A(n)/n$ exhibits substantial oscillations, 
these fluctuations cancel 
when the contributions from the two time scales are combined.

We define a metric $\mathbf{d}$ on $X$ by 
\[ \mathbf{d}(x, x^\prime) = 2^{-\min\{n\geq 1\mid x_n\neq x^\prime_n\}}. \]
We assume $\mathbf{d}(x, x^\prime) = 0$ if $x=x^\prime$.
Define a metric $\mathbf{d}^\prime$ on $Y$ similarly.

Let $0<\varepsilon<1/2$.
It is straightforward to see that 
\[ 2^{A(n)}\cdot 4^{w A(n)} \leq \#^w(\Omega, n, \varepsilon) \leq
   \mathrm{Const_\varepsilon}\cdot 2^{A(n)}\cdot 4^{w A(n)}. \]
In particular, for $w=1/2$, 
\[ 2^{2A(n)} \leq \#^{1/2}(\Omega, n, \varepsilon) \leq 
 \mathrm{Const}_{\varepsilon} \cdot 2^{2A(n)}. \]
Therefore, we obtain
\[ \hinf^{1/2}(\Omega, T)  = \frac{2}{3} \log 2, \quad 
    \h^{1/2}(\Omega, T) =  \frac{4}{3} \log 2. \]

The next proposition is the main result of this section.

\begin{proposition} \label{prop: FH entrpy of a strange example}
$\FHh^{1/2}(\Omega, T) = \log 2$. Consequently,
\begin{equation*}
  \hinf^{1/2}(\Omega, T)  = \frac{2}{3} \log 2 
   < \FHh^{1/2}(\Omega, T)  = \log 2 < 
    \h^{1/2}(\Omega, T) =  \frac{4}{3} \log 2. 
\end{equation*}    
\end{proposition}

\begin{proof}
First we prove $\FHh^{1/2}(\Omega, T) \leq \log 2$. 
Let $\varepsilon>0$ and $N\in \mathbb{N}$.
We take a natural number $L = L(\varepsilon)$ so large that $2^{-L} < \varepsilon$.
We define a finite subset $\mathcal{C} \subset \Omega$ by
\[ \mathcal{C} = \left\{(a_n, b_n)_{n\in \mathbb{N}} \in \Omega 
\middle|\, a_n = 0 \text{ for $n>\left\lceil\frac{N}{2}\right\rceil+L$, and }
b_n = 0 \text{ for $n> N +L$}\right\}. \]
We have 
\[ \Omega \subset \bigcup_{x\in \mathcal{C}} B^{1/2}_{N}(x, \varepsilon). \]
The cardinality of $\mathcal{C}$ is equal to 
$2^{A(N + L)}\cdot 4^{A(\lceil N/2 \rceil+L)}$.
For $s:= \log 2$, we have
\[
  \Lambda^{1/2, s}_{N, \varepsilon}(\Omega)  \leq \sum_{x\in \mathcal{C}} e^{-s N} 
   = 2^{A(N + L)}\cdot 4^{A(\lceil N/2 \rceil+L)}\cdot 2^{-N} 
   = 2^{A(N + L) + 2 A(\lceil N/2 \rceil+L) - N}.
\]
The right-most side is bounded by a constant depending only on 
$L = L(\varepsilon)$ by Lemma \ref{lemma: combinatorial property of I} (2).
Therefore, letting $N\to \infty$, we obtain
$\Lambda^{1/2, s}_{\varepsilon}(\Omega) < \mathrm{Const}_{\varepsilon}$.
This shows $\FHh^{1/2}(\Omega, T, \varepsilon) \leq s = \log 2$.
Letting $\varepsilon \to 0$, we obtain $\FHh^{1/2}(\Omega, T) \leq \log 2$.

Next we prove $\FHh^{1/2}(\Omega, T) \geq \log 2$.
We fix $1/4 < \varepsilon < 1/2$.
Let $x = (a_n, b_n)_{n\in \mathbb{N}}$ and 
$x^\prime = (a_n^\prime, b_n^\prime)_{n\in \mathbb{N}}$ be two points of $X$.
We have $x^\prime \in B^{1/2}_N(x, \varepsilon)$ if and only if
\[ a_n = a_n^\prime \text{ for $1\leq n \leq \lceil N/2\rceil$}, \quad 
   b_n = b_n^\prime \text{ for $1\leq n \leq N$}. \]
It follows that 
\begin{itemize}
  \item if $B^{1/2}_N(x, \varepsilon) 
  \cap B^{1/2}_{N^\prime}(x^\prime, \varepsilon) \neq \emptyset$ with $N\leq N^\prime$, 
  then $B^{1/2}_{N^\prime}(x^\prime, \varepsilon) \subset B^{1/2}_N(x, \varepsilon)$;
  \item $B^{1/2}_N(x, \varepsilon) = B^{1/2}_{N}(x^\prime, \varepsilon)$ 
  for any $x^\prime \in B^{1/2}_N(x, \varepsilon)$.
\end{itemize}

We call a non-empty set of the form $B^{1/2}_N(x, \varepsilon)\cap \Omega$
$(x\in \Omega)$ a \textbf{level $N$ weighted cylinder of $\Omega$}.
The number of level $N$ weighted cylinders of $\Omega$ is equal to
$2^{A(N)}\cdot 4^{A(\lceil N/2\rceil)}$.

Suppose we are given a covering 
\[ \Omega \subset \bigcup_{i=1}^k B^{1/2}_{n_i}(x_i, \varepsilon). \]
We want to estimate $\sum_{i=1}^k e^{-s n_i}$ with $s = \log 2$ from below.
Since weighted cylinders are either disjoint or one is contained in the other, 
we can assume that $B^{1/2}_{n_i}(x_i, \varepsilon)$ $(1\leq i \leq k)$ 
are disjoint.
Moreover we can assume $x_i\in \Omega$ $(1\leq i \leq k)$ because, 
after discarding balls disjoint from $\Omega$, we may choose
each center in $\Omega$ without changing the corresponding ball.

Set $N= \max\{n_i \mid 1\leq i \leq k\}$.
The number of level $N$ weighted cylinders of $\Omega$ 
contained in $B^{1/2}_{n_i}(x_i, \varepsilon)$ is equal to
$2^{A(N) - A(n_i)}\cdot 4^{A(\lceil N/2\rceil) - A(\lceil n_i/2\rceil)}$.
Thus 
\[ \sum_{i=1}^k 2^{A(N) - A(n_i)}\cdot 
   4^{A(\lceil N/2\rceil) - A(\lceil n_i/2\rceil)} 
  =  2^{A(N)}\cdot 4^{A(\lceil N/2\rceil)}. \]
Therefore we obtain
\[ \sum_{i=1}^k 2^{-A(n_i)}\cdot 4^{-A(\lceil n_i/2\rceil)} = 1. \]
This is a kind of \lq\lq{}Kraft identity\rq\rq{} adapted to our situation; see
\cite[\S 5.2]{Cover--Thomas}.

By Lemma \ref{lemma: combinatorial property of I} (2), we have
$A(n_i) + 2 A(\lceil n_i/2\rceil) \geq n_i$.
Hence 
\[ \sum_{i=1}^k 2^{-n_i} \geq 1. \]
Thus $\Lambda^{1/2, \log 2}_{1, \varepsilon}(\Omega) \geq 1$.
This has shown $\FHh^{1/2}(\Omega, T, \varepsilon) \geq \log 2$.
Since $\FHh^{1/2}(\Omega,T,\varepsilon)$ is nondecreasing
as $\varepsilon$ decreases, we obtain
\[
\FHh^{1/2}(\Omega,T)
\geq
\FHh^{1/2}(\Omega,T,\varepsilon)
\geq\log2.
\]
\end{proof}

\end{document}